%% file: parrep_arxiv.tex
\documentclass[12pt]{amsart}
\usepackage{amssymb}
\usepackage{amsmath}
\usepackage{mathrsfs}
\usepackage{amsthm}
\usepackage{graphicx}
\usepackage{url}
\usepackage{subfigure}
\usepackage{verbatim}
\usepackage{hyperref}
\usepackage{cite}
\usepackage{color}

\numberwithin{equation}{section}
\newtheorem{thm}{Theorem}[section]

\newtheorem{prop}{Proposition}[section]

\newtheorem{lem}{Lemma}[section]

\include{gsmacros}

\begin{document}

\title{Numerical Analysis of Parallel Replica Dynamics}
\author{Gideon Simpson} \email{gsimpson@umn.edu}
\author{Mitchell Luskin}\email{luskin@umn.edu}

\date{\today}
\maketitle

\begin{abstract}
  Parallel replica dynamics is a method for accelerating the
  computation of processes characterized by a sequence of infrequent
  events.  In this work, the processes are governed by
  the overdamped Langevin equation.  Such processes spend much of
  their time about the minima of
  the underlying potential, occasionally transitioning into different
  basins of attraction. The essential idea of parallel replica
  dynamics is that the exit time distribution from a given well for a
  single process can be approximated by the minimum of the exit time
  distributions of $N$ independent identical processes, each run for
  only $1/N$-th the amount of time.

  While promising, this leads to a series of numerical analysis
  questions about the accuracy of the exit distributions.
  Building upon the recent work in \cite{Bris:2011p13365}, we prove a unified
  error estimate on the exit distributions of the algorithm against an
  unaccelerated process. Furthermore, we study a dephasing mechanism,
  and prove that it will successfully complete.
\end{abstract}

\section{Introduction}

Parallel replica dynamics (ParRep) is a numerical tool first
introduced by Voter in \cite{Voter:1998p13729} (see
also\cite{Voter:2002p12678,perez2009accelerated}) for accelerating the
simulation of stochastic processes characterized by a sequence of
infrequent, but rapid, transitions from one state to another.  A
standard and important problem in which such a separation of scales is
present is the migration of defects through a crystalline lattice; see
\cite{perez2009accelerated} and references therein for examples.

Roughly, the idea behind parallel replica dynamics is as follows.
Suppose a trajectory spends time $t$ in a particular state, before
transitioning into another.  Furthermore, assume $t$ is large,
relative to the scale of the time step discretization.  We wish to
avoid directly simulating a single realization for time $t$.  We
approximate the simulation of a single trajectory for time $t$ with
$N$ independent copies, each simulated for time $t/N$, and follow the
particular trajectory that escapes first.  This holds out
the promise for a linear speedup with the number of independent
realizations we are able to simulate.

Of course, this is not exact, and error is introduced.  A particular
concern is error in the exit distributions of the system as it
migrates from one state to another -- does ParRep disrupt the state to
state dynamics?  Inspired by the tools proposed in
\cite{Bris:2011p13365}, we prove an error estimate on the exit
distributions over a single ``cycle'' of ParRep (the transition from
one state to the next).

\subsection{The Algorithm}

We assume the system we wish to accelerate evolves according to the
overdamped Langevin equation,
\begin{equation}
  \label{e:odlang}
  dX_t = - \grad V(X_t) dt + \sqrt{2 \beta^{-1}} dB_t, \quad X_t \in \R^{n},
\end{equation}
where $B_t$ is a Wiener process and $\beta$ is proportional to inverse
temperature.  Though ParRep was originally developed for the
Langevin equations, it is readily adapted to this problem.

We next assume that our system is such that $V$ has a denumerable set
of local minima, $x_j$, $j=1,2,\ldots$ For each minima, we associate a
set $W_j\subset \R^{n}$, the ``well.''  $W_j$ could be the basin of
attraction of $x_j$; if $y(t)$ solves the ODE
\[
\dot y= - \grad V(y),\quad y(0) = y_0\in \R^n,
\]
then
\[
W_j = \set{y_0 : \lim_{t\to + \infty} y(t) = x_j}.
\]
However, this definition is not essential; for the sake of our
analysis, $W_j$ need only be a bounded set in $\R^n$ with sufficiently
regular boundary.

This motivates defining the well selection function,
\begin{equation}
  \label{e:wellfunction}
  \mathcal{S} :\R^{n} \to \mathbb{N},
\end{equation}
which identifies the basin associated with a given position.
Associated with this is the ``coarse grained'' trajectory,
\begin{equation}
  \label{e:coarse_trajectory}
  \mathcal{S}_t \equiv \mathcal{S}(X_t)
\end{equation}
which only identifies the present well.

If the wells are ``deep'' with well-defined minima, then $X_t$ will
infrequently transition from one to another.  Such a well corresponds
to a {\it metastable} state.  Much of the simulation
time will be spent waiting for a jump to occur.  The goal of ParRep is
to reduce this computational expense by providing a
satisfactory approximation of the form
\begin{equation}
  \mathcal{S}_t \approx \mathcal{S}^{\rm ParRep}_t.
\end{equation}
In other words, we are willing to sacrifice information about where the
trajectory is within each well, for the sake of rapidly computing the
sequence of wells the trajectory visits.

We now describe the ParRep algorithm in the following steps: the
decorrelation step; the dephasing step; and the parallel step.  These
steps are diagrammed in Figures \ref{f:parrep} and \ref{f:dephasing}.
We assume that the reference process $X_t^\refe$ enters well $W_j$ at
time $\tsimu$.

\begin{figure}

  \includegraphics[width=8cm]{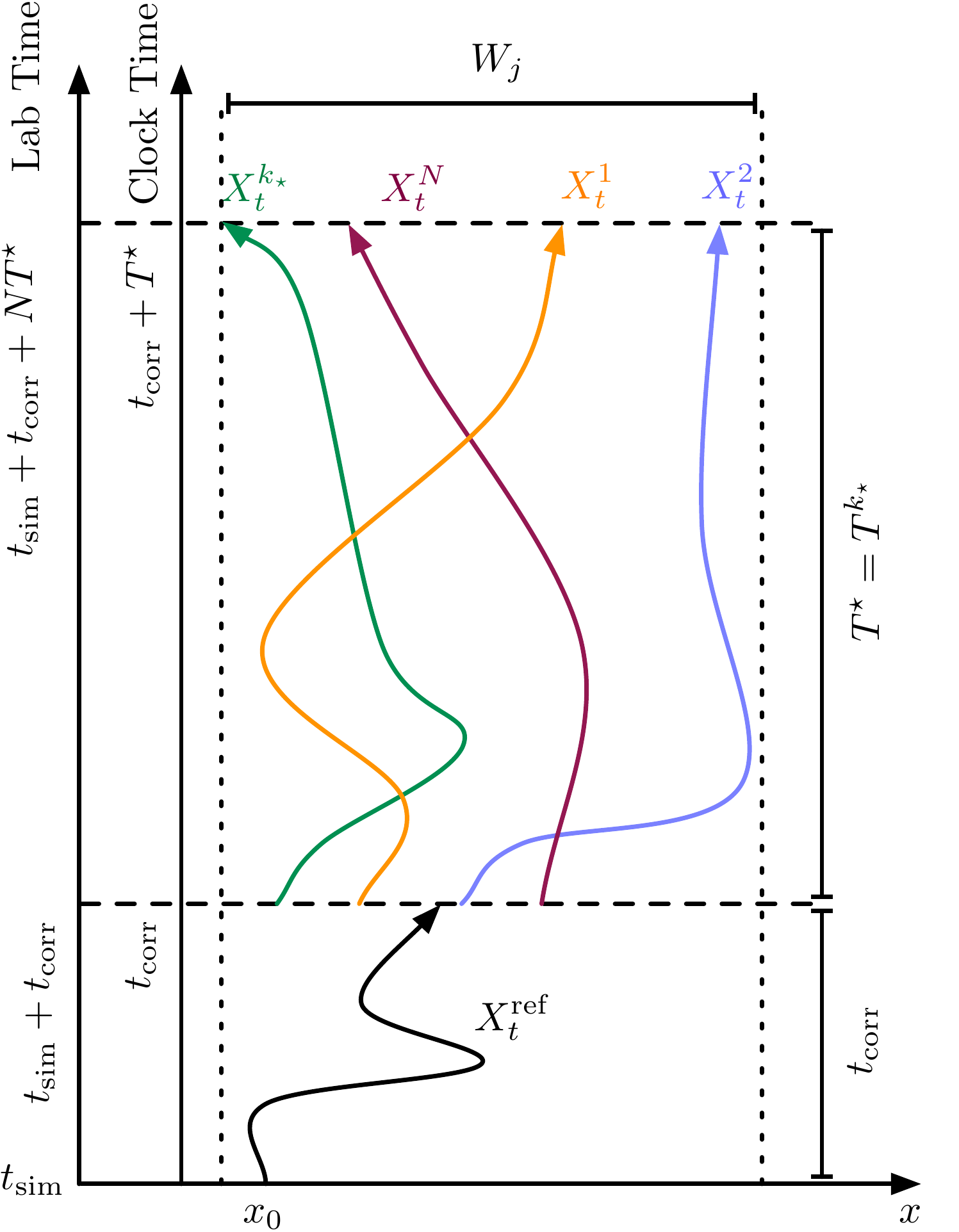}
  \caption{An illustration of the decorrelation and parallel steps of
    the ParRep algorithm in the case that the reference walker never
    leaves well $W_j$. $X^{k_\star}_t$ is the first process to exit
    the well, doing so at the computer time $ \tcorr + T^{\star}$.
    This is then translated into the lab, or physical, time $ \tsimu +
    \tcorr + N T^{\star}$.  See Figure \ref{f:dephasing} for an
    illustration of a dephasing step.}
  \label{f:parrep}
\end{figure}

\begin{figure}
  \includegraphics[width=6cm]{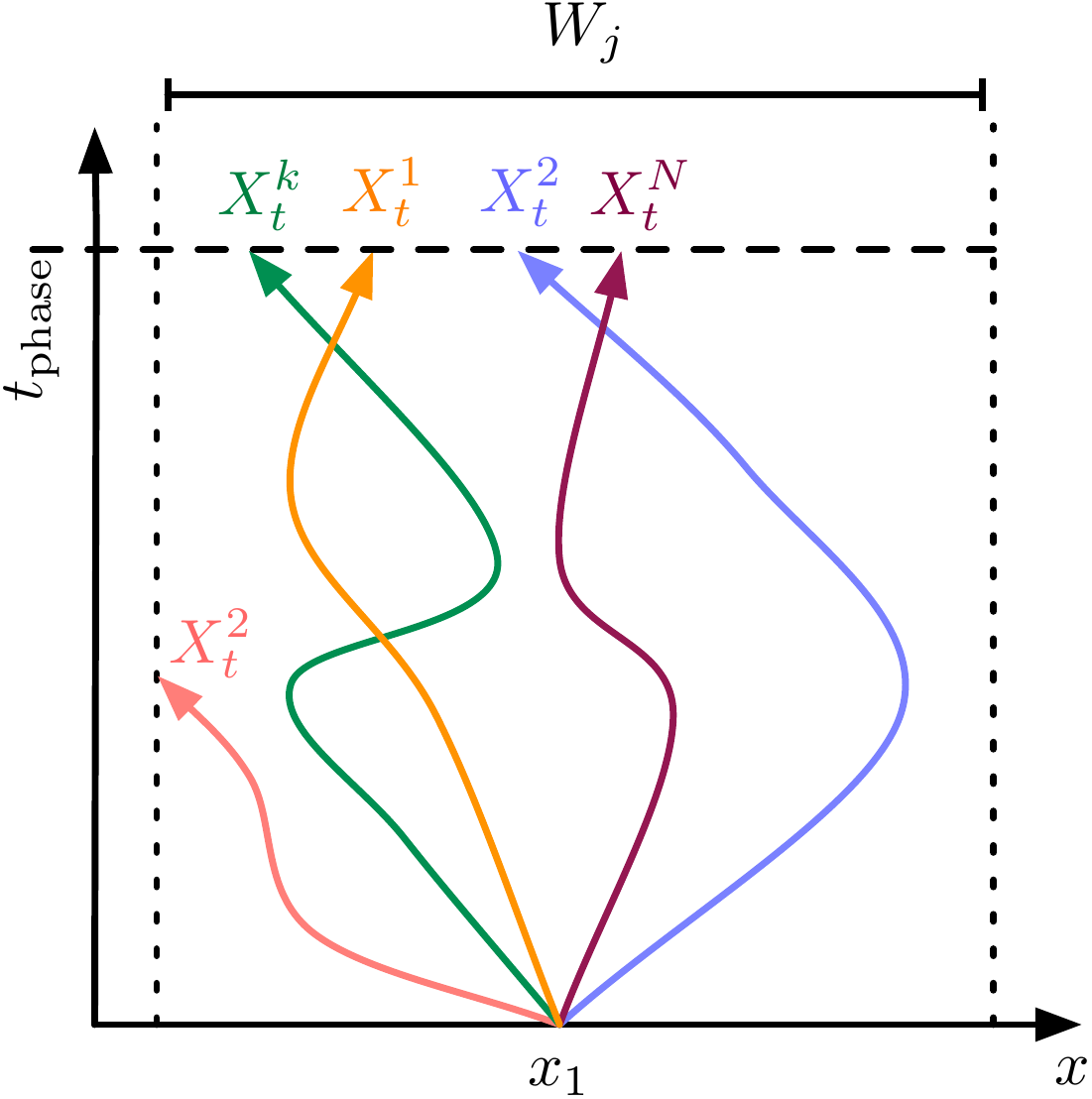}
  \caption{An illustration of a dephasing step for the ParRep
    algorithm.  In this implementation, the replicas all start from
    the same position; $\mu_{\rm phase}^0 = \delta_{x_1}$.  When
    $X^2_t$ leaves before $\tphase$, it is relaunched from the same
    position.}
  \label{f:dephasing}
\end{figure}

\begin{enumerate}
\item {\it Decorrelation Step}: Let $X_t^{\refe}$ evolve under
  \eqref{e:odlang} for $t_{\simu}\leq t\leq t_{\simu}+{\tcorr}$.
  \begin{itemize}
  \item If
    \[
    \mathcal{S}(X_t^{\refe}) = \mathcal{S}(X_{t_\simu}^{\refe})
    \]
    for all $t_{\simu}\leq t\leq t_{\simu} + {\tcorr} $, then time
    advances to $t_\simu + {\tcorr}$ and proceed.

  \item Otherwise, denote the first exit time from the well,
    \begin{equation*}
      T = \inf\set{t
        \mid \mathcal{S}(X_{t_\simu + t}^{\refe}) \neq
        \mathcal{S}(X_{t_\simu}^{\refe}) }
    \end{equation*}
    and time advances to $t_\simu + T$.  Return to the beginning of the
    decorrelation step in the new well.

  \end{itemize}

\item {\it Dephasing Step:} In conjunction with the decorrelation
  step, we launch $N$ replicas with starting positions drawn from
  distribution $\mu_{\phase}^{0}$.  These are run for $\tphase$ amount
  of time, the {\it dephasing time}.  If at any time before $\tphase$
  a replica leaves the well, it is restarted.  A replica has
  successfully dephased if it remains in the well for all of $\tphase$.

  At the completion of the decorrelation and dephasing steps, assuming
  the reference walker has not exited, we have $N$ independent walkers
  with the same distribution.  We discard the reference process.  If
  at any time during the dephasing process the reference walker
  leaves the well, the dephasing process terminates and the replicas
  are discarded.

\item {\it Parallel Step}: We now let the $N$ replicas evolve
  independently and define
  \begin{subequations}
    \begin{align}
      k_\star & = \underset{k}{\rm argmin}\; T^k,\\
      X^\star_t & = X^{k_\star}_t,\\
      T^\star &= T^{k_\star}.
    \end{align}
  \end{subequations}
  The system advances to the next well:
  \begin{subequations}
    \begin{gather}
      t_\simu \mapsto  t_\simu + \tcorr + N T^\star \\
      X^\refe_{ t_\simu + \tcorr + N T^\star } = X^{\star}_{
        T^\star}.
    \end{gather}
  \end{subequations}
  Finally, we return to the decorrelation step.

\end{enumerate}

This is a different dephasing algorithm than described in
\cite{Bris:2011p13365}.  There, after the decorrelation step, the
replicas are initiated at the the position of the reference process
and run for $\tphase$. The simulation clock is not advanced, and
replicas are replaced as need be should they exit the well.  Our
implementation has the advantage that no processor sits idle.

The reader may wonder why we would want to have a distinguished
reference process -- why not relaunch the reference process, as we
would a replica, should it exit?  We retain this feature to allow for
realizations where the process is in a well for a very short period,
far less than the decorrelation time.  These correlated events, such as {\it
  recrossings}, appear in serial simulations and should be preserved.  One may also ask why we
discard the reference process.  This is to simplify the analysis, as
it permits us to declare that the $N$ replicas are drawn from the same
distribution when the parallel step begins.

In addition to the choice of $\tcorr$ and $\tphase$, there is also the
question of what $\mu_{\phase}^0$ should be.  Again, there is
significant flexibility.  One possibility is to allow the reference
process to evolve for some amount of time, and then
the replicas could be launched from its position.  A method used in
practice is to find a local minima associated with the well, and
initiate the replicas from that position, \cite{Perez:fk}.  We
emphasize that the dephasing mechanism need not depend on any
information associated with the reference process.

In principle, ParRep offers a nearly linear speedup with the number of
independent replicas, provided $\tcorr$ is short relative to the
typical exit time.  With the explosion in the availability of
distributed computing clusters, parallel replica dynamics is an
attractive tool for studying infrequent event processes.

\subsection{Main Results}

The essential aspects of a process undergoing infrequent events are
\begin{itemize}
\item How often does it transition from one state to another?
\item What state does it transition to?
\end{itemize}
These properties are captured in $\mathcal{S}_t$. To assess how well
$\mathcal{S}_t^{\rm ParRep}$ approximates it, we are motivated to
first consider the exit distribution of a process, and how well it is
preserved.  In \cite{Bris:2011p13365}, the authors proposed a rigorous
framework in which to study ParRep. The purpose of this study is to
unify those ideas and assess the total error, over a single cycle of
ParRep, as a function of the parameters.

{\bf Note:} For brevity, we shall now take $t_\simu = 0$ and $W_j =
W$.  Throughout our paper, we shall assume:
\begin{itemize}
\item $W\subset \R^n$ is bounded;
\item $\partial W$ is sufficiently smooth;
\item $V$ is sufficiently smooth on $\overline{W}$.
\end{itemize}
Though $W$ need not correspond to a basin of
attraction, we shall continue to call it a well.

To motivate our results, we introduce some important objects.  Let
$\mu_t$ denote the law of $X_t$, conditioned on having not
left the well:
\begin{equation}
  \label{e:lawXt}
  \mu_t(A) = \prob^{\mu_0}\bracket{X_t \in A\mid T> t} =
  \frac{\prob^{\mu_0}\bracket{X_t \in A, T>
      t}}{\prob^{\mu_0}\bracket{T> t}}.
\end{equation}
The above expression is the probability of finding the processes,
$X_t$, in the set $A\subset W$, at time $t$, conditioned on the exit
time from the well, $T$, being beyond $t$, and $X_0$ being initially
distributed by $\mu_0$.  Additional details on our notation are given
below, in Section \ref{s:notation}. Under certain assumptions, the
limit
\begin{equation}
  \label{e:yaglom}
  \lim_{t\to \infty} \mu_t = \nu,
\end{equation}
exists. $\nu$ is the {\it quasistationary distribution} (QSD) and
characterizes the long term survivors of \eqref{e:odlang} in well
$W$. The properties of $\nu$ are reviewed for the reader below in
Section \ref{s:prelim}.

In the following theorems, we shall refer to ``admissible
distributions.''  This class is quite broad and includes the Dirac
distribution.  It is defined and explored in subsequent
sections. First, we have the following result on the convergence of
the exit distribution of $X_t.$
\begin{thm}[Convergence to the QSD]
  \label{t:decorr_err_soft}
  Assume $\mu_0$ is admissible.  There exist positive constants
  $\lambda_2 > \lambda_1$, $C$ and $\underline{t}$, such that for all
  $t \geq \underline{t}$ and bounded and measurable $f(\tau,
  \xi):\R^+\times \partial W\to \R$ we have
  \begin{equation*}
    \abs{\E^{\mu_t}\bracket{ f(T, X_{T})} -\E^{\nu}
      \bracket{f(T,
        X_{T})}  } \leq C {\norm{f}_{L^\infty}}  e^{-(\lambda_2 - \lambda_1) t}.
  \end{equation*}
  The constant $C$ is independent of $t$ and $f$.
\end{thm}
Taking $t$ sufficiently large so as to make this small corresponds to
the satisfactory completion of the decorrelation step; this reflects
\eqref{e:yaglom}.  We give a more precise statement of this theorem at
the beginning of Section \ref{s:decorr}, after introducing some
additional notation in Section \ref{s:prelim}.  This result also plays
a role in studying the dephasing step.
The constants $C$ and $\underline{t}$ depend on $\mu_0$, $V,$ and the
geometry of the well.  We will use the notation $C_\phase$ and
$C_\corr$, and $\underline{t}_\phase$ and $\underline{t}_\corr$ to
distinguish the constants induced by the dephasing and decorrelation
steps.

The next result ensures that the dephasing step terminates successfully:
\begin{thm}[Dephasing Process]
  \label{t:dephasing_soft}
  For an admissible distribution $\mu^0_{\phase}$ and $t_\phase \geq \underline{t}_\phase$:
  \begin{enumerate}
  \item Dephasing produces $N$ independent replicas with distributions
    $\mu_{\phase}$;
  \item Given any $\eps>0$, by taking $\tphase\geq
    \underline{t}_\phase$,
    \begin{equation*}
      \abs{\E^{\mu_\phase}\bracket{ f(T^k, X^k_{T^k})} -\E^{\nu}
        \bracket{f(T,
          X_{T})}  } \leq C_\phase e^{-(\lambda_2 - \lambda_1)\tphase}\norm{f}_{L^\infty}
    \end{equation*}
  \item The expected number of times a replica is relaunched is
    finite.
  \end{enumerate}
\end{thm}

Next, the error in the parallel step cascading from the dephasing step
can be controlled:
\begin{thm}[Parallel Error]
  \label{t:par_err_soft}
  Given $\tphase\geq \underline{t}_{\phase}$, let
  \[
  \eps_\phase \equiv C_\phase e^{-(\lambda_2 - \lambda_1)\tphase},
  \]
  and assume the dephasing step has produced $N$ i.i.d. replicas drawn
  from distribution $\mu_\phase$.

  Then the exit time converges to an exponential law, with parameter
  $N\lambda_1$,
  \begin{equation*}
    \abs{\prob^{\mu_\phase}\bracket{T^\star > t} - e^{-N \lambda_1 t} }
    \leq N\eps_\phase(1+\eps_\phase)^{N-1} e^{-N\lambda_1 t} .
  \end{equation*}

  If we additionally assume that
  $N\eps_\phase(1+\eps_\phase)^{N-1}<1$, then the hitting point 
  distribution is asymptotically independent of the exit time
  \begin{equation*}
    \abs{\prob^{\mu_\phase}\bracket{X_{T^\star}^\star \in A  \mid{T^\star >t}}
      - \int_A d\rho} \lesssim\frac{ N^2\eps_\phase(1+\eps_\phase)^{N-1}}{1-N\eps_\phase(1+\eps_\phase)^{N-1}},
  \end{equation*}
  where $\rho$ is the hitting point density and $A\subset \partial W$.
\end{thm}
Thus, for $\tphase$ large enough, we achieve
the ideal factor of $N$ speedup and we do not disrupt the hitting
point distribution too much.  The reader may find the $N$ dependence
in the error terms to be disconcerting, but it can easily be
controlled by taking $\tphase \gtrsim \log N/(\lambda_2 - \lambda_1)$.
We will return to this in the discussion.  We also note that there is
a slight abuse of notation in the above expressions.  The
superscripts, $\nu$ and $\mu_{\phase}$, should be interpreted as
$N$-tensor products, with a distinct realization drawn for each
replica.

A more detailed statement of this theorem, with explicit constants, is
given at the beginning of Section \ref{s:par}.  The hitting point
density $\rho$ is defined by \eqref{e:bdry_meas}.

However, Theorem \ref{t:par_err_soft} is only a comparison between
the parallel step and the QSD.  Our final result is a comparison
between the ParRep algorithm, including decorrelation, dephasing and
parallel steps, with an unaccelerated, serial process:
\begin{thm}[ParRep Error]
  \label{t:unified_soft}

  Let $X^{\rm s}_t$ denote the unaccelerated (serial) process and
  $X^{\rm p}_t$ denote the ParRep process, and let both the serial process
  and the reference process be initially distributed under $\mu_0$,
  an admissible distribution.  Furthermore, assume the replicas are
  initialized from $\mu_\phase^0$, also an admissible distribution.

  Given $\tcorr \geq \underline{t}_\corr$ and $\tphase \geq
  \underline{t}_\phase$, let
  \begin{align*}
    \eps_\corr & = C_\corr e^{-(\lambda_2 - \lambda_1)\tcorr},\\
    \eps_\phase & = C_\phase e^{-(\lambda_2 - \lambda_1)\tphase}.
  \end{align*}

  Letting $T^{\rm s}$ and $T^{\rm p}$ denote the physical exit times,
  we have
  \begin{equation*}
    \begin{split}
      &\abs{\prob^{\mu_0}\bracket{{ T^{\rm s}>t}}-
        \prob^{\mu_0}\bracket{{T^{\rm p}>t}}}\\
      &\quad \lesssim\bracket{\eps_\corr +
        N\eps_\phase(1+\eps_\phase)^{N-1}}e^{-\lambda_1(t-\tcorr)_+}.
    \end{split}
  \end{equation*}
  If, in addition, $\tcorr$ is sufficiently large that $\eps_\corr
  <1$, then for $A\subset \partial W$,
  \begin{equation*}
    \begin{split}
      &\abs{\prob^{\mu_0}\bracket{X_{T^{\rm s}}^{\rm s} \in A \mid {
            T^{\rm s}>t}}-
        \prob^{\mu_0}\bracket{X_{T^{\rm p}}^{\rm p}\in A\mid {T^{\rm p}>t}}}\\
      &\quad \lesssim \frac{\eps_\corr +
        N^2\eps_\phase(1+\eps_\phase)^{N-1}}{1-\eps_\corr}
    \end{split}
  \end{equation*}
\end{thm}
Thus, over a single cycle, the error in ParRep can be approximately
decomposed as
\begin{equation}
  \text{Error} = \text{Decorrelation error} + \text{Parallel error}(\text{Dephasing error}),
\end{equation}
where we view the parellel error as a function of the dephasing error.
The speedup can be seen when $T^\para$ is given further consideration.
When $T^\para > \tcorr$, $T^\para = NT^\star+ \tcorr$ where $T^\star$
is the exit time of the particular replica which escapes first.  There
will be no speedup if the exit is before $\tcorr$.

\subsection{Outline of the Paper}
In section \ref{s:prelim}, we review some important results for
\eqref{e:odlang}.  Our main Theorems are proven in Sections
\ref{s:decorr}, \ref{s:dephasing}, and \ref{s:par}.  We then discuss
our results in Section \ref{s:disc}.  Some additional calculations appear
in the appendix.

\subsection{Notation}
\label{s:notation}
Random variables, such as the position, $X_t$, and the exit time from
the well, $T$, will appear in capital letters.  Deterministic values,
such as $x$, $t$, $\tcorr$, {\it etc.} will be lower case.  We will
frequently use indicator functions in our analysis, which we write as
$1_A$, with $A$ indicating the set on which the value is one.

We are often interested in probabilities and expectations of solutions
of $X_t$ solving \eqref{e:odlang}, and its exit time $T$ from some
region $W$.  When we write
\begin{equation*}
  \E^x \bracket{f(T, X_{T})} \text{ or }
  \prob^x \bracket{T \geq t}=\E^x \bracket{1_{T\geq t}}
\end{equation*}
the superscript $x$ indicates that $x$ is the initial condition of
$X_t$; $X_0 = x$, and the expectation and probability are then taken
with respect to the underlying Wiener measure of $B_t$.

When $X_0$ is given by some distribution $\mu_0$ over $W$, we write
\begin{equation*}
  \E^{\mu_0} \bracket{f(T, X_{T})} \equiv \int_W \E^x \bracket{f(T, X_{T})} d\mu_0(x).
\end{equation*}
When we write a conditional expectation with respect to distribution
$\mu_0$, we mean
\begin{equation*}
  \E^{\mu_0} \bracket{f(T, X_{T})\mid T > t}
   \equiv \frac{\E^{\mu_0} \bracket{f(T, X_{T})1_{ T > t}} }{\prob^{\mu_0} \bracket{ T > t} }.
\end{equation*}

For the reader more accustomed to the computational physics literature,
\begin{equation*}
  \E^{\mu_0}\bracket{\mathcal{O}(X_t)} = \left\langle \mathcal{O}(t) \right \rangle.
\end{equation*}
It is helpful to explicitly include the starting distribution, $\mu_0$
associated with the process $X_t$, to avoid any ambiguity.

When we write $f \lesssim g$, we mean that there exists a constant $C>
0$ such that $f \leq C g$, but that the constant is not noteworthy.

\subsection{Acknowledgements}
The authors wish to thank D.~Aristoff, K.~Leder, S.~Mayboroda,
A. Shapeev, and O.~Zeitouni for helpful conversations in developing
these ideas.

We also thank D. Perez and A.F. Voter for conversations at LANL that
motivated important refinements of our estimates.

This work was supported by the NSF PIRE grant OISE-0967140 and the DOE
grant DE-SC0002085.

\section{Preliminary Results}
\label{s:prelim}

Before proceeding to our main results on ParRep, we review some
important results on the overdamped Langevin equation.  These results
are where our regularity assumptions on $V$, $W,$ and $\partial W$ are
needed.

Two essential tools in our study of \eqref{e:odlang} are the
Feynman-Kac formula and the quasistationary distribution, which we
briefly review here; see \cite{Bris:2011p13365} for additional
details.  First, let us recall the Feynman-Kac formula which relates solutions
of a parabolic equation with corresponding elliptic operator
\begin{equation}
  \label{e:linop}
  L \equiv - \nabla V \cdot \nabla + \beta^{-1} \Delta
\end{equation}
to solutions of \eqref{e:odlang}.

\begin{prop}[Proposition 1 of \cite{Bris:2011p13365}]
  On the parabolic domain $W \times \R^+$, let $v$ solve
  \begin{subequations}
    \label{e:fk_parabolic}
    \begin{align}
      \partial_t v &= L v,\\
      v\mid_{\partial W} & = \phi: \partial W \to \R, \\
      v(t=0) & = v_0: W \to \R.
    \end{align}
  \end{subequations}
  Then,
  \begin{equation}
    \label{e:fk}
    v(t,x) = \E^x\bracket{1_{{T}\leq t } \phi(X_{T})   } + \E^x
    \bracket{1_{T > t} v_0(X_t)  }.
  \end{equation}
\end{prop}

To say a bit more about the elliptic operator $L$, recall the
invariant measure of \eqref{e:odlang}:
\begin{equation}
  \label{e:invariant_meas}
  d\mu \equiv  Z^{-1} \exp\paren{- \beta V(x)}dx,
\end{equation}
where $Z$ is the appropriate normalization.  We introduce the Hilbert
space $L^2_\mu$, with inner product
\begin{equation}
  \label{e:mu_innerproduct}
  \inner{f}{g}_\mu \equiv \int f g d\mu.
\end{equation}
An elementary calculation shows that $L$ is self adjoint and negative
definite with respect to this inner product when supplemented with
homogeneous Dirichlet boundary conditions on $\partial W$.  Standard functional
analysis and elliptic theory tell us that $L$ has infinitely many
eigenvalue/eigenfunction pairs $(\lambda_k, u_k)$; the eigenvalues can be
ordered
\[
0 > - \lambda_1 > - \lambda_2 \geq -\lambda_3 \geq \ldots;
\]
and the eigenfunctions form a complete orthonormal basis for
$L^2_\mu(W)$.  In addition, the ground state, $u_1$, is unique and
positive.  For details, see, for example,
\cite{evans2002pde,gilbarg2001elliptic,haroske2008distributions}.  The
$\lambda_1$ and $\lambda_2$ appearing in our theorems are precisely
the first two eigenvalues.

When solving \eqref{e:fk_parabolic} with $\phi = 0$, the solution can
be expressed as
\begin{equation}
  \label{e:series_soln}
  v(x,t) = \sum_{k=1}^\infty e^{-\lambda_k t} \inner{v_0}{u_k}_\mu u_k.
\end{equation}
Out of this spectral problem, we build the norm
\begin{equation}
  \label{e:Hmu}
  \norm{f}_{H_\mu^{s}}^2 \equiv \sum_{k=1}^\infty \lambda_k^{s} \abs{\inner{f}{u_k}_\mu}^2.
\end{equation}
This generalizes to measures
\begin{equation}
  \label{e:Hmu0}
  \norm{\mu_0}_{H_\mu^{s}}^2 \equiv \sum_{k=1}^\infty \lambda_k^{s}
  \abs{\int u_k d\mu_0}^2,
\end{equation}
and to sequences, $\mathbf{a} = (a_1, a_2,\ldots)$
\begin{equation}
  \label{e:Hmuseq}
  \norm{\mathbf{a}}_{H_\mu^{s}}^2 \equiv \sum_{k=1}^\infty \lambda_k^{s}
  \abs{a_k}^2.
\end{equation}
If $\mu_0$ has an Radon-Nikodym derivative with respect to $\mu$,
\eqref{e:Hmu} and \eqref{e:Hmu0} agree.  We then define the function
spaces,
\begin{equation}
  H^{s}_\mu = \set{v \in \mathscr{S}(W)'\mid \norm{v}_{H^{s}_\mu}<\infty },
\end{equation}
where $\mathscr{S}$ is the set of smooth functions with support in
$W$, and $\mathscr{S}'$ is its dual.  We also define the projection
operator, $P_{\mathcal{I}}$, where $\mathcal{I} \subset \mathbb{N}$,
\begin{equation}
  \label{e:proj}
  P_{\mathcal{I}} f = \sum_{k\in \mathcal{I}} \inner{f}{u_k}_{\mu} u_k.
\end{equation}

Having introduced these spaces and norms, we can now clarify what was
meant by the term {\it admissible distribution} used in
the introduction.  In this work, a distribution will
be admissible with respect to $W$ if $\supp \mu_0 \subset W$, and for
some $s\geq 0$, $\norm{\mu_0}_{H^{-s}_\mu}<\infty$.

The aforementioned quasistationary distribution of \eqref{e:odlang}
associated with the set $W$ is closely related to the spectral
structure of $L$.  The QSD, $\nu$, is a time independent probability
measure satisfying, for all measurable $A\subset W$ and $t>0$:
\begin{equation}
  \nu(A) = \frac{\int_W \prob^x \bracket{X_t \in A,\; t < T} d\nu
  }{\int_W \prob^x \bracket{ t < T} d\nu} = \prob^\nu \bracket{X_t \in
    A\mid t < T}.
\end{equation}
The QSD measure $\nu$ exists and
\begin{prop}[Proposition 2 of \cite{Bris:2011p13365}]
  \begin{equation}
    \label{e:qsd_u1}
    d\nu = \frac{u_1 d\mu}{\int_W u_1 d\mu} = \frac{u_1 e^{-\beta V}dx.
    }{\int_W u_1 e^{-\beta V} dx }
  \end{equation}
\end{prop}
We refer the reader to, amongst others,
\cite{Cattiaux:2009um,Cattiaux:2010fx,Collet:1995ts,Martinez:1994vn,Martinez:2004we,Steinsaltz:2007cy}
for additional details on the QSD.  The utility of the QSD stems from
the property that if $X_0$ is distributed according to $\nu$, then:
\begin{prop}[Proposition 3 of \cite{Bris:2011p13365}]
  \label{p:qsdproperties}
  Let $\phi: \partial W \to \R$ be smooth.  Then for $t>0$
  \begin{equation*}
    \begin{split}
      \E^\nu \bracket{1_{T< t} \phi(X_{T})} =\prob^{\nu}\bracket{T <t}
      E^{\nu}\bracket{\phi(X_{T})}= (1-e^{-\lambda_1
        t})\int_{\partial_W} \phi \,d\rho
    \end{split}
  \end{equation*}
  where the exit density is given by
  \begin{equation}
    \label{e:bdry_meas}
    d\rho = -\frac{1}{\lambda_1\beta} \nabla \frac{d\nu}{dx} \cdot
    {\bf n}\, dS_x = -\frac{\nabla (u_1 e^{-\beta V} )\cdot
      {\bf n}}{\lambda_1 \beta \int_W u_1 e^{-\beta V} dx} \, dS_x,
  \end{equation}
  with ${\bf n}$ the outward pointing normal and $dS_x$ the surface
  measure.
\end{prop}
In words, $T$ is exponentially distributed with parameter $\lambda_1$,
and the first hitting point is independent of the first hitting
time. Being initially distributed according to $\nu$ is, in a sense,
ideal.  As shown by Proposition 5 of \cite{Bris:2011p13365}, were this
the case for the replicas, the parallel step of ParRep would be exact.
In practice, $X_0$ is never distributed by $\nu$, and it is the
propagation of this error that we explore.

Many of these quantities can be reformulated in terms of the
Fokker-Planck equation for density $p^x(t,y)$, $x\in W$,
\begin{gather*}
  \partial_t p^x= L^\ast p^x = \nabla_y \cdot\paren{p^x\nabla V +
    \beta^{-1} \nabla p^x}, \\
  p^x|_{\partial W} = 0, \quad p_0^x = \delta_x(y).
\end{gather*}
Though we will not make use of this, the reader more accustomed to
Fokker-Planck may find it helpful to re-express various quantities in
terms of $p^x$.  With regard to exit distributions,
\begin{align*}
  \E^{x}\bracket{\phi(X_T)1_{T<t}} &= \int_0^t\int_{\partial W}
  -\phi(y)\beta^{-1}\nabla p^x\cdot {\bf n}\,dS_y,\\
  \prob^{x}\bracket{t< T} & = \int_t^\infty \int_{\partial W}
  -\beta^{-1}\nabla p^x\cdot {\bf n} dS_y = \int_W p^x(t,y)\,dy.
\end{align*}
These can be integrated against the density of the QSD,
$\tfrac{d\nu}{dy}$, which solves $L^\ast \tfrac{d\nu}{dx} = -\lambda_1
\tfrac{d\nu}{dy}$, to obtain
\[
p^\nu(y,t) = e^{-\lambda_1t}\frac{d\nu}{dy}
\]
as a particular solution of the Fokker-Planck equation. This directly shows the
independence of exit time and hitting point.  Substituting into the
above integrals reproduces Proposition \ref{p:qsdproperties}.

\section{Convergence to the QSD -- Proof of Theorem
  \ref{t:decorr_err_soft}}
\label{s:decorr}

In this section we prove Theorem \ref{t:decorr_err_soft}, which
we first restate with more detail:

\begin{thm}[Convergence to the QSD]
  \label{t:decorr_err}
  Given $s \geq 0$, let $\mu_0$ be a distribution with $\supp \mu_0
  \subset W$ and $\norm{\mu_0}_{H^{-s}_\mu}<\infty$.  There
  exists
  \begin{equation}
    \label{e:tunderline}
    \underline{t} \gtrsim \set{\norm{P_{[2,\infty)}\mu_0}_{H^{-s}_\mu} / \int
      u_1d\mu_0}^{4/( n + 2s)}
  \end{equation}
  such that for all $t \geq \underline{t}$ and for all bounded and
  measurable $f(\tau, \xi):\R^+ \times \partial W \to \R$
  \begin{equation}
    \label{e:decorr_err}
    \begin{split}
      &\abs{\E^{\mu_t}\bracket{ f(T, X_{T})} -\E^{\nu} \bracket{f(T,
          X_{T})}  } \\
      &\ \lesssim {\norm{f}_{L^\infty}}\paren{\int u_1 d\mu_0}^{-1} {
        \underline{t}}^{-n/4- s/2} e^{-(\lambda_2 - \lambda_1) (t-
        \underline{t})} \norm{P_{[2,\infty)}\mu_0}_{H_\mu^{-s}}.
    \end{split}
  \end{equation}
\end{thm}
This is a refinement of Proposition 6 from \cite{Bris:2011p13365},
which now admits initial distributions which lack an $L^2$
Radon-Nikodym derivative.  Indeed, for appropriate $s$,
$\mu_0$ can be a Dirac distribution.  Though this is a parabolic flow
which will instantaneously regularize such rough data, it is essential
to an analysis of ParRep as one often wants to use Dirac mass initial
conditions.

In addition to this result, we present an extension which is essential
to obtaining the results in Section \ref{s:par} on the parallel step.

\subsection{Proof of Theorem \ref{t:decorr_err}}
\label{s:decorr_proof}

\begin{proof}
  We first write
  \begin{equation*}
    \E^{\mu_t}\bracket{f(T, X_{T})} = \int_W \E^x\bracket{ f(T,
      X_T)}   d\mu_t = \int_W F(x) d\mu_t
  \end{equation*}
  where we have defined $F(x) \equiv \E^x\bracket{ f(T,X_{T})}$.
  Thus,
  \begin{equation}
    \E^{\mu_t}\bracket{f(\Texit, X_{\Texit})} = \frac{\int_W\E^x\bracket{F(X_t)1_{\Texit> t}}d\mu_0}
    {\int_W\E^x \bracket{1_{\Texit> t}} d\mu_0}.
  \end{equation}
  Applying Feynman-Kac, \eqref{e:fk}, to this,
  \begin{equation}
    \E^{\mu_t}\bracket{f(\Texit, X_{\Texit})} = \frac{\int_W v(t,x)
      d\mu_0}{\int_W \bar v(t,x) d\mu_0}
  \end{equation}
  where $v$ solves \eqref{e:fk_parabolic} with $v_0 = F$ and $\phi =
  0$, while $\bar v$ solves it with $v_0 = 1$ and $\phi = 0$.  For
  brevity, let
  \begin{equation}
    \hat F_k =  \int Fu_k d\mu
    , \quad \hat 1_k = \int u_k d\mu, \quad
    \hat \mu_{0,k} = \int u_k d\mu_0.
  \end{equation}
  Expressing $v$ and $\bar v$ as series solutions using
  \eqref{e:series_soln}, we have
  \begin{equation}
    v(t,x)  = \sum_{k=1}^\infty e^{-\lambda_k t} \hat F_k u_k(x),
    \quad \bar v(t,x)  = \sum_{k=1}^\infty e^{-\lambda_k t} \hat 1_k u_k(x).
  \end{equation}

  After a bit of rearrangement, the error can be expressed as
  \begin{equation}
    \label{e:error1}
    \begin{split}
      e(t) &\equiv \abs{\E^{\mu_t}\bracket{f(T, X_{T})} - \E^\nu
        \bracket{f(T, X_{T}) }}\\
      &= \left|\frac{\sum_k e^{-(\lambda_k-\lambda_1) t}
          \paren{\hat F_k- \hat 1_k\int Fd\nu } \hat \mu_{0,k} }{\hat
          1_1 \hat \mu_{0,1} + \sum_k e^{-(\lambda_k-\lambda_1) t}
          \hat 1_k \hat \mu_{0,k} }\right|
    \end{split}
  \end{equation}
  where the sums are from $k=2$ to $\infty$ since $\hat F_1=\hat 1_1
  \int_W F d\nu$.  Noting that
  \begin{equation*}
    \begin{split}
      \abs{\hat F_k - \hat 1_k \int F d\nu } &\leq \int \abs{F u_k
      }d\mu + \int \abs{F}
      d\nu \int \abs{u_k} d\mu\\
      &\quad\leq 2\norm{f}_{L^\infty} \int \abs{u_k}d\mu \leq 2
      \norm{f}_{L^\infty} \sqrt{\mu(W)},
    \end{split}
  \end{equation*}
  we can rewrite the numerator as
  \begin{equation*}
    \begin{split}
      &\abs{\sum_{k=2}^\infty e^{-(\lambda_k - \lambda_1)t} \paren{
          \hat F_k -
          \hat 1_k \int F d\nu   }\hat \mu_{0,k}}\\
      &\quad\leq 2 \sqrt{\mu(W)} \norm{f}_{L^\infty} \sum_{k=2}^\infty
      e^{-(\lambda_k - \lambda_1)t} \abs{\hat \mu_{0,k}}\\
      &\quad\leq 2\sqrt{\mu(W)} \norm{f}_{L^\infty} e^{-(\lambda_2 -
        \lambda_1)(t-t_1)}\sum_{k=2}^\infty
      e^{-(\lambda_k - \lambda_1)t_1} \abs{\hat \mu_{0,k}}\\
      &\quad \leq 2 \sqrt{\mu(W)} \norm{f}_{L^\infty} e^{-(\lambda_2 -
        \lambda_1)(t-t_1)}\sum_{k=2}^\infty e^{-\kappa \lambda_k t_1}
      \abs{\hat \mu_{0,k}}
    \end{split}
  \end{equation*}
  where $\kappa = 1- \lambda_1/ \lambda_2$ and $t\geq t_1 >0$.
  Applying Proposition \ref{p:sum_bound2} from the appendix to this,
  the numerator is bounded by
  \begin{equation}
    \label{e:num_bound}
    \begin{split}
      &\abs{\sum_{k=2}^\infty e^{-(\lambda_k - \lambda_1)t} \paren{
          \hat F_k -
          \hat 1_k \int F d\nu  }\hat \mu_{0,k}}\\
      &\quad \lesssim
      \norm{P_{[2,\infty)}\mu_0}_{H^{-s}_\mu}\norm{f}_{L^\infty}
      e^{-(\lambda_2 - \lambda_1)(t-t_1)} t_1^{-n/4 - s/2}.
    \end{split}
  \end{equation}
  The constant that has been absorbed into
  the $\lesssim$ symbol is independent of $t$, $f$ and $\mu_0$.

  To ensure the denominator is uniformly bounded away from zero, we
  use a similar treatment,
  \begin{equation*}
    \begin{split}
      \sum_{k=2}^\infty e^{-(\lambda_k - \lambda_1)t} \hat 1_k \hat
      \mu_{0,k} &\leq \sqrt{\mu(W)} e^{-(\lambda_2 -\lambda_1)(t-t_2)}
      \sum_{k=2}^\infty
      e^{-\kappa\lambda_k t_2} \abs{\hat \mu_{0,k}} \\
      &\lesssim \norm{P_{[2,\infty)}\mu_0}_{H^{-s}_\mu} e^{-(\lambda_2
        - \lambda_1)t}t_2^{-n/4 - s/2}
    \end{split}
  \end{equation*}
  for $t\geq t_2 >0$, which may differ from $t_1$.  Therefore,
  \begin{equation*}
    \begin{split}
      &\hat 1_1 \hat \mu_{0,1} + \sum_{k=2}^\infty e^{-(\lambda_k -
        \lambda_1)t} \hat 1_k \hat\mu_{0,k}\\
      &\quad\gtrsim \hat 1_1 \hat \mu_{0,1} - e^{-(\lambda_2 -
        \lambda_1) (t-t_2)} t_2^{-n/4 - s/2}
      \norm{P_{[2,\infty)}\mu_0}_{H^{-s}_\mu}.
    \end{split}
  \end{equation*}
  For a sufficiently large $t \geq \underline t \geq t_2 > 0$, the
  denominator is bounded from below by
  \begin{equation}
    \label{e:denom_bound}
    \hat 1_1 \hat\mu_{0,1} + \sum_{k=2}^\infty e^{-(\lambda_k - \lambda_1)t}\hat 1_k
    \hat \mu_{0,k} \geq
    \frac{1}{2}\hat 1_1 \hat \mu_{0,1}  = \frac{1}{2}\int u_1 d\mu \int u_1
    d\mu_0 > 0.
  \end{equation}
  Roughly,
  \begin{equation}
    \underline t \gtrsim \set{\norm{P_{[2,\infty)}\mu_0}_{H^{-s}_\mu} / \int
      u_1d\mu_0}^{4/( n + 2s)}.
  \end{equation}

  Taking $t_1 = t_2 = \underline t $ in \eqref{e:num_bound} and
  \eqref{e:denom_bound} we have that for $t\geq  \underline t $
  \begin{equation*}
    \begin{split}
      e(t) &\lesssim \paren{\int u_1 d\mu_0}^{-1} e^{-(\lambda_2
        -\lambda_1)(t-\underline t)} (  \underline t )^{-n/4-
        s/2}\\
      &\quad \times \norm{f}_{L^\infty}
      \norm{P_{[2,\infty)}\mu_0}_{H^{-s}_\mu}.
    \end{split}
  \end{equation*}

  Finally, for this estimate to hold for general bounded and
  measurable $f$, we apply a density argument with respect to the
  $L^\infty$ norm.

\end{proof}

The inclusion of $\int u_1 d\mu_0$ in the preceding result is
deliberate as $\mu_0$ is, to a degree, a user specified parameter.
Moreover, $\int u_1 d\mu_0$ could be quite small.  Indeed, when a
$X_t$ first enters $W$, it is near $\partial W$ and the support of
$\mu_0$ is in a neighborhood of $\partial W$; we may have $\mu_0 =
\delta_{x}$ where $x$ is close to $\partial W$.  As $u_1$ is
continuous and vanishes on $\partial W$,
\[
\int_W u_1\delta_{x} = \bigo \paren{\dist(x,\partial W )}.
\]
We also see that as $\mu_0 \to \nu$,
$\norm{P_{[2,\infty)}\mu_0}_{H^{-s}_\mu} \to 0$, and the error
vanishes.

It remains to identify distributions and values of $s$ for which
$\norm{\mu_0}_{H^{-s}_\mu}<\infty$.  In the case that $\mu_0$ has an
$L^2_\mu$ Radon-Nikodym derivative, one readily sees that
$\norm{\mu_0}_{H^{-s}_\mu}<\infty$ for $s\leq0.$ Indeed, when $s=0$,
this results collapses onto the $L^2_\mu$ estimate of
\cite{Bris:2011p13365}. This extends to $\mu_0$ possessing $L^p_\mu$
densities for any $p \ge 2$.

For the case $\mu_0 = \delta_{x}$, a Dirac mass, we shall have that
$\mu_0 \in H^{-s}_\mu$ when $s$ is large enough to embed $H^s_\mu$
into $L^\infty$. If the $\partial W$ is sufficiently smooth, then by
standard elliptic theory, $H^s_\mu$ and $H^s$ will be equivalent for
$s \ge 0$, and we have the embedding for $s>n/2$,
\cite{gilbarg2001elliptic, evans2002pde, adams2003sobolev}.  Refined
elliptic estimates may weaken such assumptions on the boundary.

\subsection{Exit Times}
In the case that we are interested in exit times, we have a result
closely related to Theorem \ref{t:decorr_err}.
\begin{thm}
  \label{t:single_exit_est}
  Assume $\mu_0$ satisfies the assumptions of Theorem
  \ref{t:decorr_err} and $t_0 \geq \underline{t}$.  Then for $t\geq 0$,
  \begin{equation}
    \abs{\prob^{\mu_{t_0}}\bracket{T>t} - e^{-\lambda_1 t}} \leq C
    e^{-\lambda_1 t} e^{-(\lambda_2 - \lambda_1) t_0}
  \end{equation}
  where $C$ is the pre-exponential factor in
  \eqref{e:decorr_err} and is independent of $t$ and $t_0$.
\end{thm}

\begin{proof}
  As before, we rely on \eqref{e:fk} and the series expansions
  \eqref{e:series_soln} to write
  \begin{equation*}
    \prob^{\mu_{t_1}}\bracket{T>t} = \frac{\prob^{\mu_0}[T> t_1 +
      t]}{\prob^{\mu_0}[T> t_1]} = \frac{\sum_{k=1}^\infty e^{-\lambda_k
        (t + t_1) }\hat{1}_k \hat\mu_{0,k}}{\sum_{k=1}^\infty
      e^{-\lambda_k t_1}\hat{1}_k \hat\mu_{0,k}}.
  \end{equation*}
  Comparing against the QSD,
  \begin{equation*}
    \begin{split}
      \abs{\prob^{\mu_{t_1}}\bracket{T>t} - e^{-\lambda_1t} } &=
      \abs{\frac{\sum_{k=1}^\infty e^{-\lambda_k (t + t_1) }\hat{1}_k
          \hat\mu_{0,k}}{\sum_{k=1}^\infty
          e^{-\lambda_k t_1}\hat{1}_k \hat\mu_{0,k}} - e^{-\lambda_1 t}}\\
      & = \abs{\frac{\sum_{k=1}^\infty \paren{e^{-\lambda_k (t + t_1)
            } - e^{-\lambda_k t_1 - \lambda_1 t}}\hat{1}_k
          \hat\mu_{0,k}}{\sum_{k=1}^\infty
          e^{-\lambda_k t_1}\hat{1}_k \hat\mu_{0,k}} }\\
    \end{split}
  \end{equation*}
  In the numerator, the $k=1$ term vanishes, leaving
  \begin{equation*}
    e^{-\lambda_1 t} \abs{\frac{\sum_{k=2}^\infty \paren{1- e^{-(\lambda_k-\lambda_1) t}} e^{-\lambda_k t_1}\hat{1}_k \hat\mu_{0,k}}{\sum_{k=1}^\infty
        e^{-\lambda_k t_1}\hat{1}_k \hat\mu_{0,k}} }\leq e^{-\lambda_1 t}\frac{\sum_{k=2}^\infty e^{-\lambda_k t_0}\abs{\hat{1}_k \hat\mu_{0,k}}}{\abs{\sum_{k=1}^\infty
        e^{-\lambda_k t_0}\hat{1}_k \hat\mu_{0,k}}}
  \end{equation*}
  Using the same methods as in the proof of Theorem
  \ref{t:decorr_err},
  \[
  \frac{\sum_{k=2}^\infty e^{-\lambda_k t_1}\abs{
      \hat\mu_{0,k}}}{\abs{\sum_{k=1}^\infty e^{-\lambda_k
        t_1}\hat{1}_k \hat\mu_{0,k}}}\lesssim \paren{\int u_1
    d\mu_0}^{-1} {\underline{t}}^{-n/4- s/2} e^{-(\lambda_2 -
    \lambda_1) (t_0-\underline{t})}
  \norm{P_{[2,\infty)}\mu_0}_{H_\mu^{-s}}
  \]
\end{proof}
This estimate plays an important role in our analysis of ParRep.
Indeed, we will frequently confront terms of the form $\E^{\mu_{t_0}}
\bracket{f(X,T)1_{T>t}}$, and we will want to compare against the
corresponding term for the QSD.  One could naively apply Theorem
\ref{t:decorr_err} to estimate such a term, with observable $g_t(\xi,
\tau) = f(\xi, \tau)1_{\tau>t}$.  However, this is wasteful, as the
observable is going to be taken over realizations which not only have
not left the well before $t_0$, but remain in the well for at least
an additional $t$.  We thus have the following identity.
\begin{lem}
  \label{e:advancetime}
  Given $t, t_0 \geq 0$,
  \begin{equation}
    \E^{\mu_{t_0}}
    \bracket{f(X_T, T) 1_{T>t}} = \E^{\mu_{t_0 + t}}
    \bracket{f(X_T, T+t)}\prob^{\mu_{t_0}}\bracket{T>t}.
  \end{equation}
\end{lem}
\begin{proof}
  This reflects the Markovian nature of the process.  Writing out
  the lefthand side,
  \begin{equation*}
    \begin{split}
      \E^{\mu_{t_0}} \bracket{f(X_T, T) 1_{T>t}} &= \int_W \E^x
      \bracket{f(X_T, T) 1_{T>t}} \mu_{t_0}(dx)\\
      & = \frac{\int_W \E^x \bracket{f(X_T, T) 1_{T>t}}
        \prob^{\mu_0}\bracket{X_t \in dx, T>t_0}
      }{\prob^{\mu_0}\bracket{T>t_0} }
    \end{split}
  \end{equation*}
  The numerator is
  \begin{equation*}
\begin{split} \int_W \E^x \bracket{f(X_T, T) 1_{T>t}}
    \prob^{\mu_0}\bracket{X_t \in dx, T>t_0} &=
  \E^{\mu_0}\bracket{f(X_T, T-t_0)1_{T-t_0>t}1_{T>t_0}}\\
&=\E^{\mu_0}\bracket{f(X_T, T-t_0)1_{T>t_0+t}},
\end{split}
  \end{equation*}
  where $t_0$ is subtracted off to make the observable
  consistent.  The same argument shows
  \begin{equation*}
    \E^{\mu_{t_0 + t}}
    \bracket{f(X_T, T+t)} = \frac{\E^{\mu_0}\bracket{f(X_T, T-t_0)1_{T>t+t_0}}}{\prob^{\mu_0}\bracket{T>t+t_0}  }.
  \end{equation*}
  Combining these three expressions completes the proof.
\end{proof}
In principle, we can use this Lemma and Theorem
\ref{t:single_exit_est} to obtain refinements on Theorem
\ref{t:decorr_err} for observables that include $1_{T>t}$ terms.

\section{The Dephasing Step -- Proof of Theorem
  \ref{t:dephasing_soft}}
\label{s:dephasing}

We now examine our dephasing step,
\begin{thm}
  \label{t:dephasing}
  Given $s\geq 0$, assume $\supp \mu_{\phase}^0 \subset W$ and  $\norm{\mu^0_{\phase}}_{H^{-s}_\mu}<\infty$.  Then
  \begin{enumerate}
  \item The dephasing step produces $N$ independent replicas with
    distributions $\mu_{\phase}$,
    \begin{equation*}
      \mu_{\phase}(A)  = \prob^{\mu_\phase^0}\bracket{X_{\tphase} \in A\mid T>\tphase};
    \end{equation*}
  \item There exists $\underline{t}_\phase$ and $C_\phase$ such that
    for $\tphase\geq \underline{t}_\phase$,
    \begin{equation*}
      \abs{\E^{\mu_\phase}\bracket{ f(T^k, X^k_{T^k})} -\E^{\nu}
        \bracket{f(T,
          X_{T})}  } \leq \norm{f}_{L^\infty}   C_\phase
      e^{-(\lambda_2 - \lambda_1) \tphase};
    \end{equation*}
    for any bounded measurable $f:\R^+\times \partial W\to \R$ and all
    $k=1,\ldots, N$.
  \item The expected number of times a replica is relaunched is
    finite.
  \end{enumerate}
\end{thm}

To prove Theorem \ref{t:dephasing}, we must establish:
\begin{enumerate}
\item The replicas are independent and have law $\mu_\phase$;
\item The error of $\mu_\phase$ can be made small;
\item The expected number of relaunches is finite.
\end{enumerate}

The first property is obvious as each of the replicas is driven by an
independent Brownian motion, and we only retain realizations for
which $T> \tphase$.  The second property follows from
Theorem \ref{t:decorr_err}.

To prove the third property, we must establish that replicas initiated
from $\mu_\phase^0$ have a nonzero chance of surviving till $\tphase$:
\begin{lem}
  \label{l:tlaunch}
  Assume that $\mu_{\phase}^0$ satisfies the hypotheses of Theorem
  \ref{t:dephasing},
  \begin{equation*}
    \prob^{\mu_{\phase}^0}\bracket{\Texit^k \geq
      \tphase}  \equiv p > 0.
  \end{equation*}
\end{lem}

\begin{proof}
  Observe that we have the following monotonicity property for $t_2 >
  t_1$,
  \[
  0\leq \prob^{\mu_{\phase}^0}\bracket{\Texit \geq t_2} \leq
  \prob^{\mu_{\phase}^0}\bracket{\Texit \geq t_1}.
  \]

  We now argue by contradiction. Assume that at some $t_1>0$,
  $\prob^{\mu_{\phase}^0}\bracket{\Texit \geq t_1}=0$.  By the above
  monotonicity, $\prob^{\mu_{\phase}^0}\bracket{\Texit \geq t_2}=0$
  for all $t_2\geq t_1$.  Using a similar approach as in the proof of
  Theorem \ref{t:decorr_err}, we write
  \[
  \prob^{\mu_{\phase}^0}\bracket{\Texit \geq t} = \bar v(x,t) =
  \sum_{k=1}^\infty e^{-\lambda_kt} \int u_k d{\mu_{\phase}^0} \int
  u_k d\mu
  \]
  where $\bar v$ solves \eqref{e:fk_parabolic} with $v_0 =1$ and $\phi
  = 0$.  Therefore,
  \begin{equation*}
    \begin{split}
      \bar v(x,t_2) &= \sum_{k=1}^\infty e^{-\lambda_kt_2}
      \hat{\mu}_{\phase, k}^0 \int u_k d\mu = \sum_{k=1}^\infty e^{-\lambda_kt_2} \hat{\mu}_{\phase, k}^0 \hat{1}_k\\
      &\geq e^{-\lambda_1 t_2} \set{\hat{\mu}_{\phase, 1}^0 \hat{1}_1
        - e^{-(\lambda_2-\lambda_1)t_2}\sum_{k=2}^\infty e^{-\kappa
          \lambda_k t_2} \abs{\hat{1}_k} \abs{\hat{\mu}_{\phase, k}^0 }}\\
      & \gtrsim e^{-\lambda_1 t_2}\set{\hat{\mu}_{\phase, 1}^0
        \hat{1}_1 - e^{-(\lambda_2-\lambda_1)t_2} t_2^{-n/4-s/2}
        \norm{P_{[2, \infty)}{\mu}_{\phase}^0 }_{H^{-s}_\mu}}
    \end{split}
  \end{equation*}
  Then taking $t_2$ sufficiently large,
  \begin{equation*}
    \begin{split}
      &\hat{\mu}_{\phase, 1}^0 \hat{1}_1 -
      e^{-(\lambda_2-\lambda_1)t_2} t_2^{-n/4-s/2}
      \norm{P_{[2,\infty)}\delta_x}_{H^{-s}_\mu}\\
      &\geq \frac{1}{2}\hat{\mu}_{\phase, 1}^0\hat{1}_1 =
      \frac{1}{2}\int u_1d\mu_{\phase}^0 \int u_1 d\mu >0,
    \end{split}
  \end{equation*}
  since $\int u_1d\mu_{\phase}^0 >0$. Thus,  we have a contradiction.
\end{proof}
This calculation reveals a role played by the choice of
$\mu_{\phase}^0$. If concentrated near the well boundary,
${\prob^{\mu_{\phase}^0}\bracket{\Texit^k \geq \tphase}=p}$ could be
quite small.  This will induce the replicas to relaunch many times, as
the next result shows.  Thus, for computational efficiency, a
distribution concentrated deep in the well's interior is desirable.

\begin{lem}
  Assume that $\mu_{\phase}^0$ satisfies the hypotheses of Theorem
  \ref{t:dephasing}, and that $
  \prob^{\mu_{\phase}^0}\bracket{\Texit^k \geq \tphase}=p>0$.  Then
  \begin{equation*}
    \E^{\mu_{\phase}^0}\bracket{\text{Number of relaunches}} =   (1-p)/p < \infty
  \end{equation*}
\end{lem}
\begin{proof}
  The probability of relaunching $m$ times is the probability of
  exiting $m$ times and surviving on the $m+1$-th time.  Interpreting
  this in terms of $\Texit^k$ and using the assumption, $
  \prob^{\mu_{\phase}^0}[\text{$m$ relaunches}] = (1-p)^{m} p$.  Thus,
  \begin{equation*}
    \begin{split}
      \E^{\mu_{\phase}^0}\bracket{\text{Number of relaunches}}& = \sum_{m=0}^\infty m \cdot \prob^{\mu_{\phase}^0}[\text{$m$ relaunches}]  \\
      & = \sum_{m=0}^\infty m (1-p)^m p= \frac{1-p}{p}<\infty.
    \end{split}
  \end{equation*}
\end{proof}


\section{The Parallel Step -- Proofs of Theorems \ref{t:par_err_soft}
  and \ref{t:unified_soft}}
\label{s:par}
First, we restate Theorem \ref{t:par_err_soft} with additional detail:
\begin{thm}[Parallel Error]
  \label{t:par_err}
  Given $\tphase\geq \underline{t}_{\phase}$, let
  \[
  \eps_\phase \equiv C_\phase e^{-(\lambda_2 - \lambda_1)\tphase},
  \]
  and assume the dephasing step has produced $N$ i.i.d. replicas drawn
  from distribution $\mu_\phase$.  Then the exit time distribution of
  the parallel step converges to an exponential,
  \begin{equation}
    \label{e:par_Texit_err}
    \abs{\prob^{\mu_\phase}\bracket{T^\star > t} - e^{-N \lambda_1 t} }
    \leq \eps_\phase  N (1 + \eps_\phase)^{N-1} e^{-N\lambda_1 t}.
  \end{equation}

  If $\phi: \partial W \to \R$ is bounded and measurable, the exit
  distribution converges to one that is independent of exit time,
  \begin{equation}
    \label{e:par_exit_err}
    \begin{split}
      &\abs{\E^{\mu_\phase}\bracket{1_{T^\star >t}
          \phi(X_{T^\star}^\star)}
        - e^{-N \lambda_1 t} \int_{\partial W} \phi d\rho}\\
      &\quad \lesssim N^2 (1 + \eps_\phase)^{N-1}\eps_\phase
      \norm{\phi}_{L^\infty} e^{-N\lambda_1 t}.
    \end{split}
  \end{equation}

  If, in addition, $N \eps_\phase(1+\eps_\phase)^{N-1} < 1$,  then
  \begin{equation}
    \label{e:par_exit_indep_err}
    \begin{split}
      & \abs{\E^{\mu_\phase}\bracket{\phi(X_{T^\star}^\star)
          \mid{T^\star >t}}
        - \int_{\partial W} \phi d\rho}\\
      &\quad\lesssim \frac{N^2 \norm{\phi}_{L^\infty} \eps_\phase (1 +
        \eps_\phase)^{N-1}}{1 - N \eps_\phase(1+\eps_\phase)^{N-1}}.
    \end{split}
  \end{equation}
\end{thm}

\begin{proof}
  To prove \eqref{e:par_Texit_err}, we begin by writing,
  \begin{equation*}
    \begin{split}
      &\abs{\prob^{\mu_{\phase}}\bracket{ T^\star> t} -
        e^{-N\lambda_1 t}}\\
      &\quad = \abs{\Pi_{k=1}^N \prob^{\mu_{\phase}}\bracket{T^k> t}
        -\Pi_{k=1}^N \prob^{\nu}\bracket{T^k> t}  } \\
      &\quad = \abs{\prob^{\mu_{\phase}}\bracket{T^1> t}^{N}
        -e^{-N\lambda_1 t} }\\
      &\quad= \abs{\prob^{\mu_{\phase}}\bracket{T^1> t} -e^{-\lambda_1
          t} } \abs{\sum_{k=0}^{N-1} \prob^{\mu_{\phase}}\bracket{T^1>
          t}^{k}e^{-(N-1-k)\lambda_1 t} }.
    \end{split}
  \end{equation*}
  From Theorem \ref{t:single_exit_est}, we know
  \begin{equation*}
    \abs{\prob^{\mu_{\phase}}\bracket{T> t}
      -e^{-\lambda_1 t} }\leq \eps_{\phase}e^{-\lambda_1 t}.
  \end{equation*}
  Therefore,
  \begin{equation*}
      \abs{\prob^{\mu_{\phase}}\bracket{ T^\star\geq t} -
        e^{-N\lambda_1 t}}\leq \eps_\phase e^{-\lambda_1 t} N (
      1+\eps_\phase)^{N-1}
      e^{-(N-1)\lambda_1 t}.
  \end{equation*}

  To prove \eqref{e:par_exit_err}, we begin by writing the expectation
  as
  \begin{equation*}
    \begin{split}
      \E^{\mu_{\phase}} \bracket{1_{T^\star> t} \phi(X_{T^\star}^\star
        ) } &= \E^{\mu_{\phase}} \bracket{1_{T^{k_\star}> t}
        \phi(X_{T^{k_\star}}^{k_\star}) }\\
      & = \sum_{k=1}^N \E^{\mu_{\phase}} \bracket{1_{T^k> t}
        \phi(X_{T^{k}}^k) 1_{k = k_\star} }\\
      & = \sum_{k=1}^N\E^{\mu_{\phase}}\bracket{1_{T^k> t}
        \phi(X_{T^{k}}^k) \Pi_{l \neq k} 1_{T^l > T^k}1_{T^l>t} }.
    \end{split}
  \end{equation*}
  In the above expression, we have used that since $T^\star>t$, $T^l>t$
  for each $l$.  Then, using Lemma \ref{e:advancetime} on each of the
  processes,
  \begin{equation*}
    \begin{split}
      &\E^{\mu_{\phase}}\bracket{1_{T^k> t}
        \phi(X_{T^{k}}^k) \Pi_{l \neq k} 1_{T^l > T^k}1_{T^l>t} }\\
      & = \E^{\mu_{\tphase+t}}\bracket{ \phi(X_{T^{k}}^k) \Pi_{l \neq
          k}
        1_{T^l> T^k}}\Pi_{l=1}^N\prob^{\mu_\phase}[T^l>t]\\
      & = \E^{\mu_{\tphase+t}}\bracket{ \phi(X_{T^{k}}^k) 1_{k =
          k_\star}}\prob^{\mu_\phase}[T>t]^N.
    \end{split}
  \end{equation*}
  This leads to the expression
  \begin{equation}
    \E^{\mu_{\phase}} \bracket{1_{T^\star> t} \phi(X_{T^\star}^\star
      ) }  =  \E^{\mu_{t_\phase+t}} \bracket{\phi(X_{T^\star}^\star
      ) }\prob^{\mu_\phase}[T^\star>t].
  \end{equation}
  Comparing against the QSD,
  \begin{equation}
    \label{e:parhittingptdiff}
    \begin{split}
      &\abs{ \E^{\mu_{\phase}} \bracket{1_{T^\star> t}
          \phi(X_{T^\star}^\star ) } - \E^{\nu} \bracket{1_{T^\star>
            t} \phi(X_{T^\star}^\star
          ) } }\\
      &\leq \abs{ \E^{\mu_{t_\phase+t}} \bracket{
          \phi(X_{T^\star}^\star
          ) }  }\abs{\prob^{\mu_\phase}[T^\star>t]- e^{-N\lambda_1 t}}\\
      & \quad + e^{-N\lambda_1 t} \abs{ \E^{\mu_{t_\phase+t}}
        \bracket{ \phi(X_{T^\star}^\star ) }- \E^{\nu} \bracket{
          \phi(X_{T^\star}^\star ) } }.
    \end{split}
  \end{equation}
The first difference can be treated by \eqref{e:par_Texit_err}, but
the second difference requires more care.

  Given an arbitrary distribution $\eta$ for $X_0$, we define
  \begin{equation}
    \calP^{\eta}(t) \equiv \prob^{\eta}[T>
    t]=\prob^{\eta}[T^k > t],\quad k = 1\ldots N.
  \end{equation}
  Consequently, $\calP^{\nu}(t) = e^{-\lambda_1 t}$ and
  \begin{equation}
    \begin{split}
      \E^{\mu_{t_\phase+t}} \bracket{ \phi(X_{T^\star}^\star ) } &=
      \sum_{k=1}^N \E^{\mu_{t_\phase+t}} \bracket{
        \phi(X_{T}^k)\Pi_{l\neq k} 1_{T_l>T_k}}\\
      & = \sum_{k=1}^N \E^{\mu_{t_\phase+t}} \bracket{ \phi(X_{T}^k)
        \calP^{\mu_{t_\phase+t}}(T^k)^{N-1} }.
    \end{split}
  \end{equation}
  An analogous expansion can be made with $\nu$ in place of
  ${\mu_{t_\phase+t}}$.  Taking the difference of the two sums, and
  comparing term by term,
  \begin{equation}
    \label{e:summand_diff}
    \begin{split}
      &\abs{\E^{\mu_{t_\phase+t}} \bracket{ \phi(X_{T}^k)
          \calP^{\mu_{t_\phase+t}}(T^k)^{N-1} } - \E^\nu \bracket{
          \phi(X_{T}^k) \calP^\nu(T^k)^{N-1}  }}\\
      & \leq \abs{\E^{\mu_{t_\phase+t}} \bracket{ \phi(X_{T}^k)
          \calP^{\mu_{t_\phase+t}}(T^k)^{N-1} }- \E^\nu \bracket{
          \phi(X_{T}^k) \calP^{\mu_{t_\phase+t}}(T^k)^{N-1}  }}\\
      & \quad + \abs{\E^\nu \bracket{ \phi(X_{T}^k)
          \calP^{\mu_{t_\phase+t}}(T^k)^{N-1} }-\E^\nu \bracket{
          \phi(X_{T}^k) \calP^\nu(T^k)^{N-1} }}.
    \end{split}
  \end{equation}
  By Theorem \ref{t:decorr_err} the first difference in
  \eqref{e:summand_diff} is bounded by
  \begin{equation*}
    \begin{split}
      &\abs{\E^{\mu_{t_\phase+t}} \bracket{ \phi(X_{T}^k)
          \calP^{\mu_{t_\phase+t}}(T^k)^{N-1} }- \E^\nu \bracket{
          \phi(X_{T}^k) \calP^{\mu_{t_\phase+t}}(T^k)^{N-1}  }}\\
      &\quad \leq \eps_\phase e^{-(\lambda_2-\lambda_1)t}
      \norm{\phi}_{L^\infty},
    \end{split}
  \end{equation*}
  since $\mathcal{P}\leq 1$.

  For the other difference in \eqref{e:summand_diff}, we can replicate
  the proof of \eqref{e:par_Texit_err} to obtain, for any $\tau\geq
  0$,
  \begin{equation*}
    \begin{split}
      &\abs{\calP^{\mu_{t_\phase+t}}(\tau)^{N-1} - \calP^\nu(\tau)^{N-1} }\\
      & \leq (N-1)\eps_\phase e^{-(\lambda_2-\lambda_1)t}
      (1+\eps_\phase
      e^{-(\lambda_2-\lambda_1)t})^{N-2}e^{-(N-1)\lambda_1 \tau}\\
      &\leq (N-1)\eps_\phase (1+\eps_\phase)^{N-2}.
    \end{split}
  \end{equation*}
  Therefore,
  \begin{equation*}
    \begin{split}
      &\abs{\E^\nu \bracket{ \phi(X_{T}^k)
          \calP^{\mu_{t_\phase+t}}(T^k)^{N-1} }-\E^\nu \bracket{
          \phi(X_{T}^k) \calP^\nu(T^k)^{N-1}  }}\\
      & \leq \norm{\phi}_{L^\infty} \eps_\phase (N-1) (1+
      \eps_\phase)^{N-2}.
    \end{split}
  \end{equation*}
  So \eqref{e:summand_diff} can be bounded by
  \begin{equation*}
    \begin{split}
      &\abs{\E^{\mu_{t_\phase+t}} \bracket{ \phi(X_{T}^k)
          \calP^{\mu_{t_\phase+t}}(T^k)^{N-1} } - \E^\nu \bracket{
          \phi(X_{T}^k) \calP^\nu(T^k)^{N-1}  }}\\
      &\leq \norm{\phi}_{L^\infty} \eps_\phase \bracket{1 + (N-1) (1+
        \eps_\phase)^{N-2}}.
    \end{split}
  \end{equation*}

  Returning to \eqref{e:parhittingptdiff}, using
  \eqref{e:par_Texit_err} to treat the first difference and the
  preceding calculation to treat the second, we have:

  \begin{equation}
    \begin{split}
      &\abs{ \E^{\mu_{\phase}} \bracket{1_{T^\star> t}
          \phi(X_{T^\star}^\star ) } - \E^{\nu} \bracket{1_{T^\star>
            t} \phi(X_{T^\star}^\star
          ) } }\\
      & \leq N\norm{\phi}_{L^\infty}\eps_\phase e^{-N\lambda_1 t} (1 +
      \eps_\phase)^{N-1} \\
      &\quad+N\norm{\phi}_{L^\infty} \eps_\phase  e^{-N\lambda_1 t} \bracket{1 +  (N-1) (1+ \eps_\phase)^{N-2}}\\
      & \lesssim N^2 \norm{\phi}_{L^\infty} \eps_\phase e^{-N\lambda_1
        t} (1 + \eps_\phase)^{N-1}.
    \end{split}
  \end{equation}

  Finally, to prove \eqref{e:par_exit_indep_err},

  \begin{equation*}
    \begin{split}
      &\abs{ \E^{\mu_{\phase}} \bracket{\phi(X_{T^\star}^\star ) \mid
          {T^\star> t} } - \E^{\nu} \bracket{\phi(X_{T^\star}^\star
          ) \mid {T^\star> t} } }\\
      &=\abs{\frac{\E^{\mu_{\phase}} \bracket{\phi(X_{T^\star}^\star )
            1_{T^\star> t} }}{\prob^{\mu_\phase}\bracket{T^\star> t} }
        -\frac{\E^{\nu} \bracket{\phi(X_{T^\star}^\star
            ) 1_{T^\star> t} }}{\prob^\nu\bracket{T^\star> t}  }  }\\
      &\leq \abs{\E^{\mu_{\phase}} \bracket{\phi(X_{T^\star}^\star )
          1_{T^\star> t} }-\E^\nu \bracket{\phi(X_{T^\star}^\star )
          1_{T^\star> t}
        }}\frac{1}{{\prob^{\mu_\phase}\bracket{T^\star> t}  }} \\
      &\quad + \abs{{\E^{\nu} \bracket{\phi(X_{T^\star}^\star )
            1_{T^\star> t} }}}
      \frac{\abs{\prob^{\mu_\phase}\bracket{T^\star> t} -
          \prob^\nu\bracket{T^\star> t}
        }}{\prob^{\mu_\phase}\bracket{T^\star>
          t}\prob^\nu\bracket{T^\star> t} }.
    \end{split}
  \end{equation*}
  For the first difference,
  \begin{equation*}
    \begin{split}
      &\abs{\E^{\mu_{\phase}} \bracket{\phi(X_{T^\star}^\star )
          1_{T^\star> t} }-\E^\nu \bracket{\phi(X_{T^\star}^\star )
          1_{T^\star> t}
        }}\frac{1}{{\prob^{\mu_\phase}\bracket{T^\star> t}  }}\\
      & \lesssim \frac{N^2 \norm{\phi}_{L^\infty} \eps_\phase (1 +
        \eps_\phase)^{N-1}}{1 - N \eps_\phase(1+\eps_\phase)^{N-1}}.
    \end{split}
  \end{equation*}
  For the second difference,
  \begin{equation*}
    \begin{split}
      &\abs{{\E^{\nu} \bracket{\phi(X_{T^\star}^\star ) 1_{T^\star> t}
          }}} \frac{\abs{\prob^{\mu_\phase}\bracket{T^\star> t} -
          \prob^\nu\bracket{T^\star> t}
        }}{\prob^{\mu_\phase}\bracket{T^\star>
          t}\prob^\nu\bracket{T^\star> t}  }\\
      &\leq \frac{N\norm{\phi}_{L^\infty}
        \eps_\phase(1+\eps_\phase)^{N-1}}{1 - N
        \eps_\phase(1+\eps_\phase)^{N-1}}.
    \end{split}
  \end{equation*}
  Combining these estimates, we have our result.

\end{proof}

Lastly, we prove Theorem \ref{t:unified_soft}, which we first restate
with additional detail:
\begin{thm}[ParRep Error]
  Let $X^{\rm s}_t$ denote the unaccelerated (serial) process and
  $X^{\rm p}_t$ denote the ParRep process, and assume that both $X^\refe_t$ and $X^{\rm s}_t$ are initially distributed
  under $\mu_0$, an admissible distribution.  Also assume that
  $\mu_\phase^0$ is admissible.

  Given $\tcorr \geq \underline{t}_\corr$ and $\tphase \geq
  \underline{t}_{\phase}$, let
  \begin{align*}
    \eps_\corr & =C_\corr e^{-(\lambda_2 - \lambda_1)\tcorr},\\
    \eps_\phase & = C_\phase e^{-(\lambda_2 - \lambda_1)\tphase}.
  \end{align*}

  Letting $T^{\rm s}$ and $T^{\rm p}$ denote the physical times, we
  have
  \begin{equation}
    \label{e:unified_Pt_err}
    \begin{split}
      &\abs{\prob^{\mu_0}\bracket{{ T^{\rm s}>t}}-
        \prob^{\mu_0}\bracket{{T^{\rm p}>t}}}\\
      &\quad \leq {\eps_\corr e^{-\lambda_1 t} +
        \eps_\phase N (1 +\eps_\phase)^{N-1}e^{-\lambda_1(t -
          \tcorr)_+}},
    \end{split}
  \end{equation}

  \begin{equation}
    \label{e:unified_E_err}
    \begin{split}
      &\abs{\E^{\mu_0}\bracket{\phi(X_{T^{\rm s}}^{\rm s})1_{ T^{\rm
              s}>t}}-
        \E^{\mu_0}\bracket{\phi(X_{T^{\rm p}}^{\rm p}) 1_{T^{\rm p}>t}}}\\
      &\quad \lesssim
      \bracket{\eps_\corr+ \eps_\phase N^2 (1+\eps_\phase)^{N-1}
      }\norm{\phi}_{L^\infty} e^{-\lambda_1(t-\tcorr)_+}.
    \end{split}
  \end{equation}
  If, in addition, $\eps_\corr <1$, then
  \begin{equation}
    \label{e:unified_Econd_err}
    \begin{split}
      &\abs{\E^{\mu_0}\bracket{\phi(X_{T^{\rm s}}^{\rm s})\mid {
            T^{\rm s}>t}}-
        \E^{\mu_0}\bracket{\phi(X_{T^{\rm p}}^{\rm p})\mid {T^{\rm p}>t}}}\\
      &\quad \lesssim \frac{\eps_\corr + \eps_\phase
        N^2(1+\eps_\phase)^{N-1}}{1-\eps_\corr}\norm{\phi}_{L^\infty},
    \end{split}
  \end{equation}

\end{thm}

\begin{proof}
  We begin by decomposing
  \begin{equation*}
    \begin{split}
      \prob^{\mu_0}\bracket{T^\ser > t} &=
      \prob^{\mu_0}\bracket{T^\ser >
        t\mid T^\ser \leq \tcorr}\prob^{\mu_0} \bracket{T^\ser\leq \tcorr}\\
      &\quad + \prob^{\mu_0}\bracket{T^\ser > t\mid T^\ser >
        \tcorr}\prob^{\mu_0} \bracket{T^\ser> \tcorr}.
    \end{split}
  \end{equation*}
  We analogously decompose $\prob^{\mu_0}\bracket{T^\para > t}$. For
  $t\leq \tcorr$, the serial algorithm and the reference process of
  ParRep have the same law.  Hence,
  \begin{gather*}
    \prob^{\mu_0}\bracket{T^{\rm s}\leq \tcorr} =
    \prob^{\mu_0}\bracket{T^{\rm p}\leq \tcorr},\\
    \prob^{\mu_0}\bracket{T^\ser > t\mid T^\ser \leq \tcorr} =
    \prob^{\mu_0}\bracket{T^\para > t\mid T^\para \leq \tcorr}.
  \end{gather*}
  Consequently, error only manifests itself if the parallel step is
  engaged,
  \begin{equation*}
    \begin{split}
      &\abs{\prob^{\mu_0}\bracket{T^\ser > t} -
        \prob^{\mu_0}\bracket{T^\para > t} }\\
      &= \prob^{\mu_0}\bracket{T^{\rm s}>
        \tcorr}\abs{\prob^{\mu_0}\bracket{T^\ser > t\mid T^\ser
          >\tcorr} - \prob^{\mu_0}\bracket{T^\para > t\mid T^\para >
          \tcorr}}.
    \end{split}
  \end{equation*}
  Comparing against the QSD,
  \begin{equation}
    \label{e:Pt_partition}
    \begin{split}
      &\abs{\prob^{\mu_0}\bracket{T^\ser > t} -
        \prob^{\mu_0}\bracket{T^\para > t} }\\
      &\leq \prob^{\mu_0}\bracket{T^{\rm s}>
        \tcorr}\abs{\prob^{\mu_0}\bracket{T^\ser >
          t\mid T^\ser >\tcorr} - \prob^{\nu}\bracket{T > (t-\tcorr)_+} } \\
      &\quad + \prob^{\mu_0}\bracket{T^{\rm s}>
        \tcorr}\abs{\prob^{\mu_0}\bracket{T^\para > t\mid T^\para >
          \tcorr}- \prob^{\nu}\bracket{T > (t-\tcorr)_+} }.
    \end{split}
  \end{equation}
  Examining the first term,
  \begin{equation*}
    \prob^{\mu_0}\bracket{T^\ser >
      t\mid T^\ser >\tcorr}  = \prob^{\mu_\corr} \bracket{T^\ser > (t-\tcorr)_+}.
  \end{equation*}
  By assumption and Theorem \ref{t:single_exit_est}
  \begin{equation}
    \label{e:Pt_decorr_err}
    \abs{ \prob^{\mu_\corr} \bracket{T^\ser > (t-\tcorr)_+}-
      \prob^{\nu}\bracket{T > (t-\tcorr)_+} } \leq \eps_\corr e^{-\lambda_1(t-\tcorr)_+}.
  \end{equation}
  For the other term, since the exit time is beyond $\tcorr$ the
  parallel step engages.  The single reference process is replaced by
  the ensemble of $N$ replicas drawn from $\mu_\phase$, and $T^\para =
  N T^\star + \tcorr$.  Hence,
  \begin{equation*}
    \prob^{\mu_0}\bracket{T^\para >
      t\mid T^\para > \tcorr} = \prob^{\mu_\phase}\bracket{T^\star> \tfrac{1}{N}(t-\tcorr)_+}.
  \end{equation*}
  Therefore, by Theorem \ref{t:par_err}
  \begin{equation}
    \label{e:Pt_par_err}
    \begin{split}
      &\abs{\prob^{\mu_\phase}\bracket{T^\star>
          \tfrac{1}{N}(t-\tcorr)_+} -
        \prob^{\nu}\bracket{T > (t-\tcorr)_+}}\\
      &\quad \leq \eps_\phase N (1+ \eps_\phase)^{N-1}e^{-\lambda_1
        (t-\tcorr)_+}
    \end{split}
  \end{equation}
  Substituting \eqref{e:Pt_decorr_err} and \eqref{e:Pt_par_err} into
  \eqref{e:Pt_partition}, we obtain \eqref{e:unified_Pt_err}.

  To obtain \eqref{e:unified_E_err}, we again decompose as
  \begin{equation*}
    \begin{split}
      \E^{\mu_0}\bracket{\phi(X_{T^{\rm s}}^{\rm s})1_{T^\ser>t}}& =
      \E^{\mu_0}\bracket{\phi(X_{T^{\rm s}}^{\rm s})1_{T^\ser>t}\mid
        T^{\rm s}\leq\tcorr}\prob^{\mu_0}\bracket{T^{\rm s}\leq \tcorr} \\
      &\quad + \E^{\mu_0}\bracket{\phi(X_{T^{\rm s}}^{\rm s})1_{T^{\rm
            s}>t}\mid T^{\rm s}>\tcorr}\prob^{\mu_0}\bracket{T^{\rm
          s}> \tcorr}.
    \end{split}
  \end{equation*}
  and analogously decompose the ParRep expectation. Again, for $t\leq
  \tcorr$, the serial algorithm and the reference process of ParRep
  have the same law.  Thus
  \begin{equation*}
    \E^{\mu_0}\bracket{\phi(X_{T^{\rm s}}^{\rm s})1_{T^{\rm s}>t}\mid
      T^{\rm s}\leq\tcorr} = \E^{\mu_0}\bracket{\phi(X_{T^{\rm p}}^{\rm p} )1_{T^{\rm p}>t}\mid
      T^{\rm p}\leq\tcorr}.
  \end{equation*}
  Consequently,
  \begin{equation*}
    \begin{split}
      &\abs{\E^{\mu_0}\bracket{\phi(X_{T^\ser}^{\ser})1_{T^\ser>t}}
        - \E^{\mu_0}\bracket{\phi(X_{T^\para}^\para)1_{T^\para>t}} }\\
      &= \prob^{\mu_0}\bracket{T^\ser> \tcorr}
      \abs{\E^{\mu_0}\bracket{\phi(X_{T^\ser}^{\ser})1_{T^\ser>t}\mid
          T^\ser>\tcorr} -
        \E^{\mu_0}\bracket{\phi(X_{T^\para}^\para)1_{T^\para>t}\mid
          T^\para>\tcorr}}.
    \end{split}
  \end{equation*}
  Using the QSD as an intermediary,
  \begin{equation}
    \label{e:E_partition}
    \begin{split}
      &\abs{\E^{\mu_0}\bracket{\phi(X_{T^\ser}^{\ser})1_{T^\ser>t}\mid
          T^\ser>\tcorr} -
        \E^{\mu_0}\bracket{\phi(X_{T^\para}^\para)1_{T^\para>t}\mid
          T^\para>\tcorr}}\\
      &\leq \abs{\E^{\mu_0}\bracket{\phi(X_{T^\ser}^{\ser})1_{T^\ser>t}\mid T^\ser>\tcorr} - \E^{\nu}\bracket{\phi(X_{T})1_{T>(t-\tcorr)_+} }}\\
      &\quad + \abs{
        \E^{\nu}\bracket{\phi(X_{T})1_{T>(t-\tcorr)_+}
        } -
        \E^{\mu_0}\bracket{\phi(X_{T^\para}^\para)1_{T^\para>t}\mid
          T^\para>\tcorr}}.
    \end{split}
  \end{equation}
  For the first term,
  \begin{equation*}
    \begin{split}
      \E^{\mu_0}\bracket{\phi(X_{T^\ser}^{\ser})1_{T^\ser>t}\mid
        T^\ser>\tcorr} &=
      \E^{\mu_\corr}\bracket{\phi(X_{T^\ser}^{\ser})1_{T^\ser>(t-\tcorr)_+}}\\
      & = \E^{\mu_{\tcorr +(t-\tcorr)_+}
      }\bracket{\phi(X_{T^\ser}^{\ser})}\prob^{\mu_
        \corr}\bracket{{T^\ser >(t-\tcorr)_+} }.
    \end{split}
  \end{equation*}
  Hence,
  \begin{equation}
    \label{e:E_decorr_err}
    \begin{split}
      &\abs{\E^{\mu_0}\bracket{\phi(X_{T^\ser}^{\ser})1_{T^\ser>t}\mid
          T^\ser>\tcorr} -
        \E^{\nu}\bracket{\phi(X_{T})1_{T>(t-\tcorr)_+}
        }}\\
      &\lesssim
      \eps_{\corr}\norm{\phi}_{L^\infty}e^{-\lambda_1(t-\tcorr)_+}.
    \end{split}
  \end{equation}
  For the other term, since the parallel step has engaged,
  \begin{equation}
    \E^{\mu_0}\bracket{\phi(X_{T^\para}^\para)1_{T^\para>t}\mid
      T^\para>\tcorr} = \E^{\mu_\phase}\bracket
      {\phi(X_{T^\star}^\star)1_{T^\star>\frac{1}{N}(t-\tcorr)_+}}.
  \end{equation}

  By Theorem \ref{t:par_err},
  \begin{equation}
    \label{e:E_par_err}
    \begin{split}
      &\abs{
        \E^{\nu}\bracket{\phi(X_{T})1_{T>(t-\tcorr)_+}
        } -
        \E^{\mu_0}\bracket{\phi(X_{T^\para}^\para)1_{T^\para>t}\mid
          T^\para>\tcorr}}\\
      & = \abs{
        \E^{\nu}\bracket{\phi(X_{T})1_{T>(t-\tcorr)_+}
        } -
        \E^{\mu_\phase}\bracket{\phi(X_{T^\star}^\star)1_{T^\star>\frac{1}{N}(t-\tcorr)_+}}}\\
      &\leq \eps_\phase N^2 \norm{\phi}_{L^\infty}
      (1+\eps_\phase)^{N-1} e^{-\lambda_1 (t-\tcorr)_+}.
    \end{split}
  \end{equation}
  Using \eqref{e:E_decorr_err} and \eqref{e:E_par_err} in
  \eqref{e:E_partition} gives \eqref{e:unified_E_err}.

  \eqref{e:unified_Econd_err} is proved using the preceding estimates,
  \begin{equation}
    \label{e:Econd_est1}
    \begin{split}
      &\abs{\E^{\mu_0}\bracket{\phi(X_{T^{\rm s}}^{\rm s})\mid {
            T^{\rm s}>t}}-
        \E^{\mu_0}\bracket{\phi(X_{T^{\rm p}}^{\rm p})\mid {T^{\rm p}>t}}}\\
      &\leq \abs{\frac{\E^{\mu_0}\bracket{\phi(X_{T^{\rm s}}^{\rm
              s})1_{ T^{\rm s}>t}}-
          \E^{\mu_0}\bracket{\phi(X_{T^{\rm p}}^{\rm p})1_{T^{\rm p}>t}}}{\prob^{\mu_0}\bracket{T^\ser>t}}}\\
      &\quad + \abs{ \E^{\mu_0}\bracket{\phi(X_{T^{\rm p}}^{\rm
            p})1_{T^{\rm
              p}>t}}}\abs{\frac{\prob^{\mu_0}\bracket{T^\ser>t} -
          \prob^{\mu_0}\bracket{T^\para>t}}{\prob^{\mu_0}\bracket{T^\ser>t}\prob^{\mu_0}\bracket{T^\para>t}}}\\
      &\lesssim \bracket{\eps_\corr + \eps_\phase
        N^2(1+\eps_\phase)^{N-1}}\norm{\phi}_{L^\infty}\frac{e^{-\lambda_1(t-\tcorr)_+}\prob^{\mu_0}\bracket{T^\ser>\tcorr}
      }{\prob^{\mu_0}\bracket{T^\ser>t}}\\
      &\quad+ \bracket{ \eps_\corr + \eps_\phase N (1 +
        \eps_\phase)^{N-1}}\norm{\phi}_{L^\infty}\frac{e^{-\lambda_1(t-\tcorr)_+}\prob^{\mu_0}\bracket{T^\ser>\tcorr}
      }{\prob^{\mu_0}\bracket{T^\ser>t}}\\
      &\lesssim \bracket{\eps_\corr + \eps_\phase
        N^2(1+\eps_\phase)^{N-1}}\norm{\phi}_{L^\infty}\frac{e^{-\lambda_1(t-\tcorr)_+}\prob^{\mu_0}\bracket{T^\ser>\tcorr}
      }{\prob^{\mu_0}\bracket{T^\ser>t}}.
    \end{split}
  \end{equation}
Since $(t-\tcorr)_+ + \tcorr \geq t$,
\[
\prob^{\mu_0}\bracket{T^\ser > (t-\tcorr)_+ + \tcorr}\leq
\prob^{\mu_0}\bracket{T^\ser > t}.
\]
Therefore,
\[
\frac{\prob^{\mu_0}\bracket{T^\ser>\tcorr}}{\prob^{\mu_0}\bracket{T^\ser>t}}\leq
\frac{\prob^{\mu_0}\bracket{T^\ser>\tcorr}}{\prob^{\mu_0}\bracket{T^\ser
    > (t-\tcorr)_+ + \tcorr}} =
\frac{1}{\prob^{\mu_\corr}\bracket{T^\ser > (t-\tcorr)_+}}
\]
and
\[
\frac{e^{-\lambda_1(t-\tcorr)_+}\prob^{\mu_0}\bracket{T^\ser>\tcorr}
      }{\prob^{\mu_0}\bracket{T^\ser>t}}\leq  \frac{1}{1-\eps_\corr}.
\]
  Substituting this estimate into \eqref{e:Econd_est1} yields
  \eqref{e:unified_Econd_err}.

\end{proof}

\section{Discussion}
\label{s:disc}

We have proven several theorems on the convergence of the exit
distributions of parallel replica dynamics to the underlying
unaccelerated problem.  We have also demonstrated the effectiveness of
a dephasing algorithm done in conjunction with the decorrelation
step. However, there remain several problems associated with ParRep,
both in fully justifying it as an algorithm, and implementing it in
practice.

\subsection{Error Estimates}
As we pointed out in the introduction, the error estimates in Theorem
\ref{t:par_err_soft} and Theorem \ref{t:unified_soft} include terms
which grow as $N\to \infty$.  If we take
\[
\tphase \gtrsim k_\phase\frac{\log N}{\lambda_2 - \lambda_1}
\]
for some multiplier, $k_\phase$, then the most egregious term in the estimates
is bounded by
\begin{equation*}
  \begin{split}
    \lim_{N\to \infty} N^2 \eps_\phase (1 + \eps_\phase)^{N-1} & \leq
    \lim_{N\to \infty} C_\phase e^{-k_\phase/2} \paren{1 +
      e^{-k_\phase}{C_\phase}/{N}}^{N-1}\\
    & \quad = e^{-k_\phase/2} e^{C_\phase e^{-k_\phase}}.
  \end{split}
\end{equation*}
Hence, taking $k_\phase$ large enough, the error can be made arbitrarily
small.  In contrast, the decorrelation error is independent of $N$,
and reducing the decorrelation error will not correct for the error
due to more replicas.

The error estimate on the exit time in Theorem \ref{t:par_err_soft} is
a bit deceiving and merits additional comment.  It would appear that
when we consider this cumulative distribution function at any $t>0$,
then, sending $N\to \infty$, the error vanishes.  This is a reflection
on the estimate being an absolute error.  Dividing out by
$e^{-N\lambda_1 t}$ lets us evaluate the relative error, which we see
is uniformly bounded in $t$.

We also remark that since
\[
\E[T] = \int_0^\infty \prob[T>t] dt,
\]
we can obtain error estimates on the expected exit time.  Using the
estimates in Theorem \ref{t:par_err_soft}, we see that provided $N
\eps_\phase (1+\eps_\phase)^{N-1}<1$, we have
\begin{equation}
  \abs{\E^{\mu_\phase}[T^\star]-\frac{1}{N\lambda_1}}\leq N
  \eps_\phase (1+\eps_\phase)^{N-1}.
\end{equation}
Similarly, using the estimates in Theorem \ref{t:unified_soft},
\begin{equation}
  \abs{\E^{\mu_0}[T^{\rm s}]-\E^{\mu_0}[T^{\rm p}] }\lesssim \eps_\corr
  + N \eps_\phase(1+\eps_\phase)^{N-1}.
\end{equation}

It remains to be determined whether our estimates are sharp -- is the
growth in $N$ real, or an artifact of our analysis?  While we cannot
yet address the sharpness, a simple numerical experiment indicates that
there is growth in the error as $N$ increases.    Consider the problem
\begin{equation}
\label{e:harmonicwell}
dX_t = -4 X_t dt + \sqrt{2} dB_t
\end{equation}
for the well $W=[-1,1]$, and suppose we launch $N$ replicas from the Dirac
distribution $X_0 = .1$.  By symmetry, we know that if we had perfect
dephasing, then during the parallel step
\[
\prob^{\nu}[X^\star_{T^\star} = 1] = \prob^{\nu}[X^\star_{T^\star} = -1]=
\tfrac{1}{2}.
\]
But if we incompletely dephase, then, because of our asymmetric
initial condition, we expect a higher probability of escaping at $1$
than $-1$.  For this problem, we can compute by
spectral methods that $\lambda_1 \approx 0.971972 $ and $\lambda_2
\approx 8.98262$.

To test our conjecture, that the error increases with $N$, we ran 10000
realizations of the dephasing and parallel steps with values of $N=100,
200, \ldots 1000$.  We employed Euler-Maruyama time stepping with
$\Delta t = 10^{-4}$.  We then ran this with with $t_\phase = .05$, .1
and .2.  The results appear in Figure \ref{f:harmonicwell}.

\begin{figure}
 \subfigure[$\tphase=.05$]{\includegraphics[width=6cm]{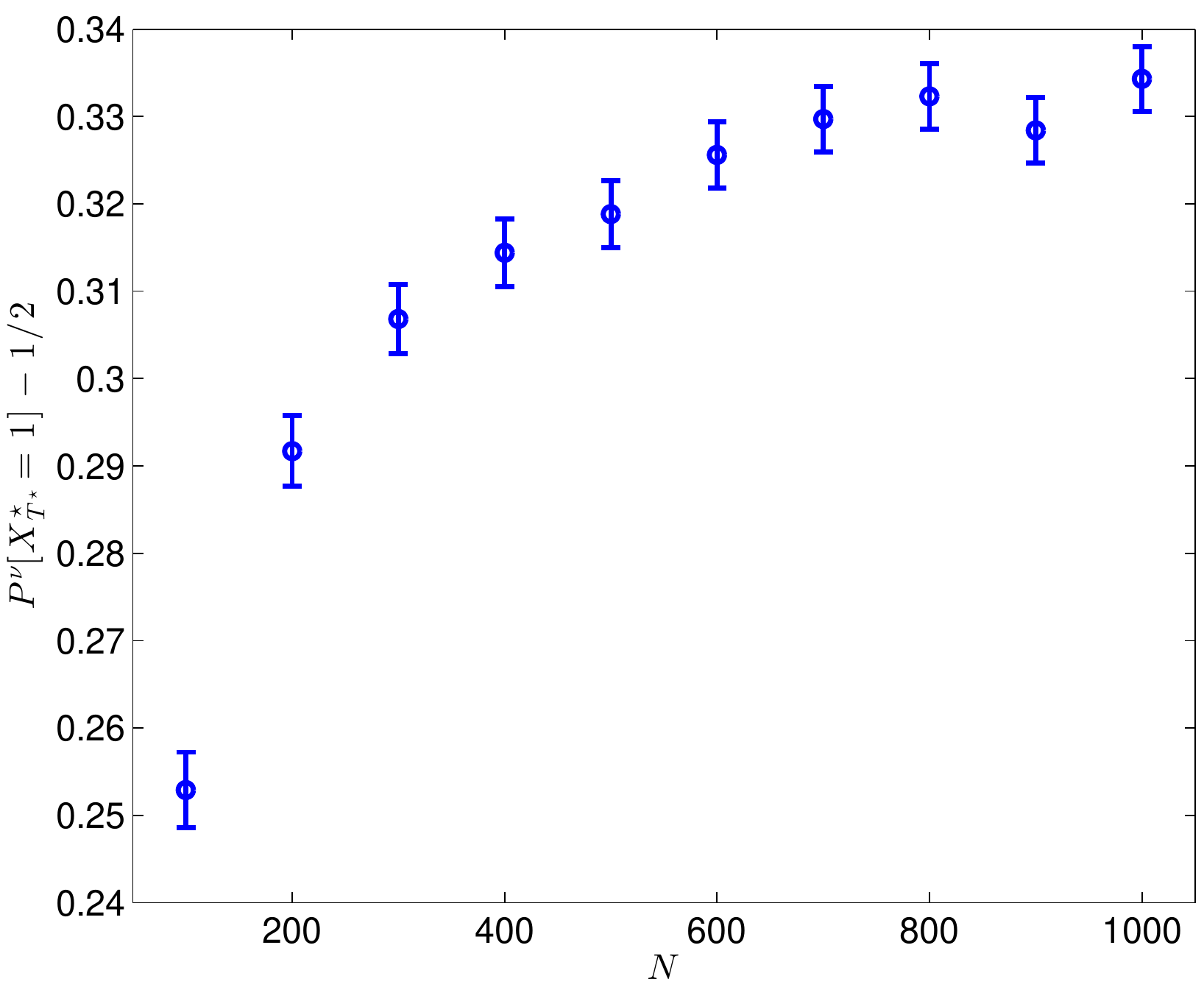}}
 \subfigure[$\tphase=.1$]{\includegraphics[width=6cm]{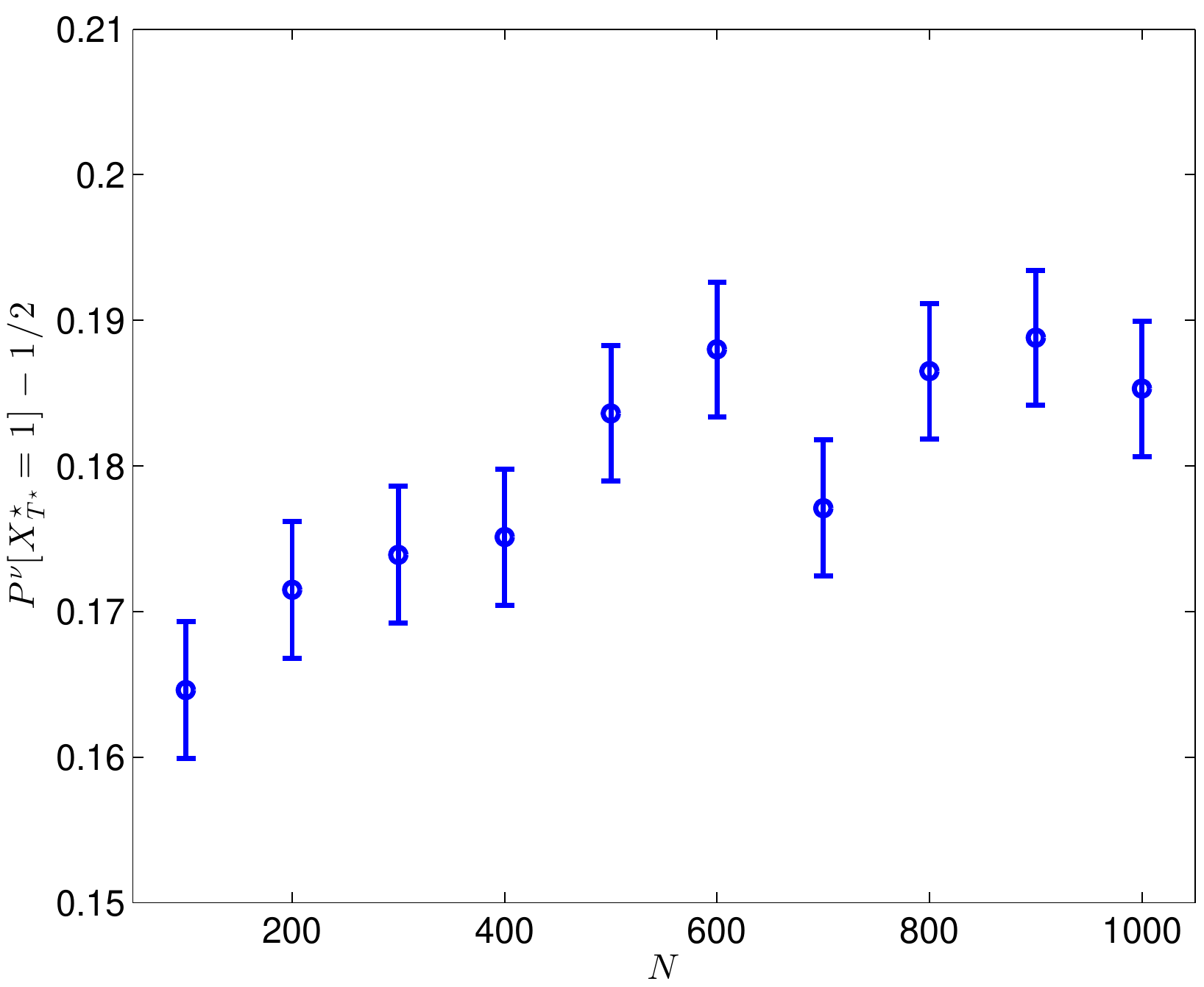}}
 \subfigure[$\tphase=.2$]{\includegraphics[width=6cm]{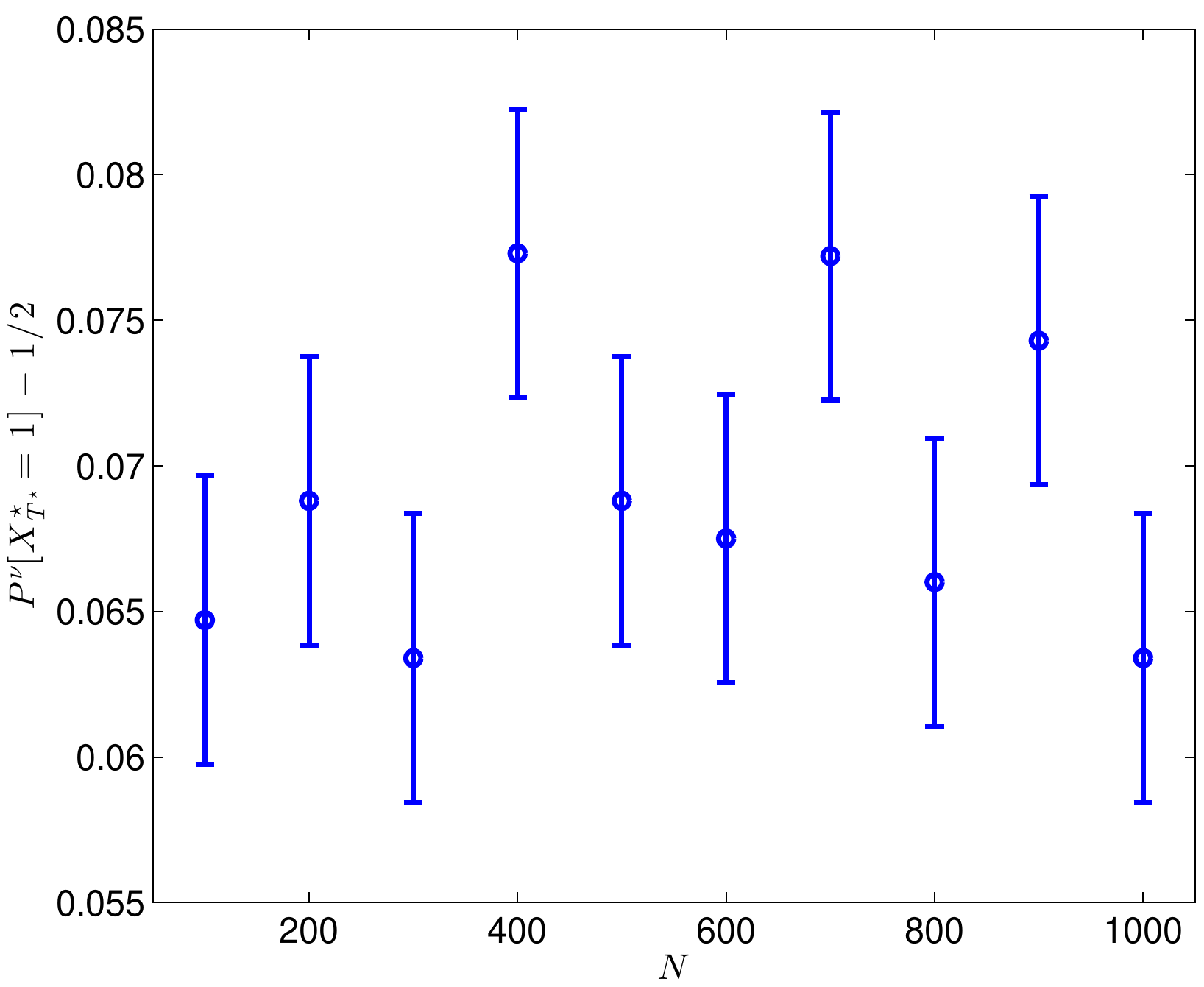}}
\caption{Three experiments on the impact of imperfect dephasing for
  \eqref{e:harmonicwell}.  With perfect dephasing, the probability of
  exiting at $x=1$ would be $1/2$, but because the initial condition
  is $.1$ and the dephasing time is finite, there is a persistent
  bias and growth in the error as $N$ increases.  95\%
  confidence intervals are plotted for 10000 realizations of each
  value of $N$.}
\label{f:harmonicwell}
\end{figure}

As we predicted, the errors decrease as $\tphase$ increases.  For the
smallest dephasing time, we also see the error increase with $N$.  At
$\tphase = .1$, there is still some increase in the error as $N$
increases, though it is less dramatic.  When $\tphase = .2$, the trend
appears to have been lost to numerical error and sampling
variability.

\subsection{Numerical Parameters \& Eigenvalues}
An essential question is how to choose of the dephasing and
decorrelation time parameters.  Based on the arguments in the
preceding section, roughly, if we desire the errors from decorrelation
and dephasing to be of the same order, then,
\[
2\log (N)\tcorr \sim \tphase.
\]
So, while they should not be the same, if we can estimate one, we can
infer the other.  There will also be some mismatch due to different
starting distributions for the reference process and the dephasing
replicas.

$\tcorr$ must be large enough so as to be representative of the QSD
while remaining computationally efficient.  Taking too large a value
of $\tcorr$ will just replicate the serial implementation with no
acceleration.  Theorem \ref{t:decorr_err} provides some insight,
already discussed in \cite{Bris:2011p13365}.  The error of $\mu_\corr$
is controlled by the following quantities:
\begin{itemize}
\item The $\mu_0$ initial distribution,
\item$\norm{P_{[2,\infty)}\mu_0}_{H^{-s}_\mu}$, the mismatch between
  the initial distribution of the reference process and the
  quasistationary distribution, $\nu$;
\item The value of $ \underline t $;
\item $\int u_1 d \mu_0$;
\item $\lambda_2 - \lambda_1$ , the spectral gap between the first two
  eigenvalues.
\end{itemize}
Based on these quantities, and how they relate to $\tcorr$, to make
the decorrelation error small, we would certainly need
\begin{equation}
  \label{e:tphase_est}
  \tcorr \gtrsim \frac{\ln \bracket{\paren{\int u_1 d\mu_0}^{-1} \norm{P_{[2,\infty)}\mu_0}_{H^{-s}_\mu} }}{\lambda_2 -
    \lambda_1} +  \underline t .
\end{equation}

The eigenvalues also play an important role in determining which problems
would benefit from ParRep is an outstanding problem, which is
an outstanding issue. For ParRep
to be efficient, we need
\begin{equation}
  \label{e:corr_constraint1}
  \tcorr \ll \E^{\mu_{\corr}}[\Texit]\sim \E^{\nu}[\Texit] = \frac{1}{\lambda_1}.
\end{equation}
This is desirable because, in the event $X_t$ does not leave the well
during the decorrelation step, it is will now take a comparatively
long time to exit. In \cite{Bris:2011p13365}, the authors suggested
\[
\tcorr \leq \E^{\mu_0}[\Texit].
\]
However, this can be problematic, depending on $\mu_0$. As previously
discussed, if the replicas launch from a position too close to the
boundary, $\E^{\mu_0}[\Texit]$ might be rather small.  This is
mitigated as $\tcorr$ becomes larger, leading to
$\E^{\mu_{\corr}}[\Texit]$ approaching the escape time of the QSD,
$\lambda_1^{-1}$.

We can see from \eqref{e:tphase_est} and constraint
\eqref{e:corr_constraint1} that ParRep will be most effective when
\begin{equation}
  \label{e:gap_cond}
  \frac{1}{\lambda_2 - \lambda_1}\ll \frac{1}{\lambda_1},
\end{equation}
or, alternatively, when $\lambda_1 \ll \lambda_2$.  Under these
conditions, $\mu_\corr$ converges to $\nu$ much more rapidly than we
expect $X_t$ to exit $W$.  \eqref{e:gap_cond} can also be viewed as a
characterization of when $W$ corresponds to a metastable state for
\eqref{e:odlang}.

Computing $\lambda_1$ and $\lambda_2$ directly from a discretization
of the elliptic operator L is intractable for all but the lowest
dimensional systems.  Instead, one must use Monte Carlo methods, such
as those found in
\cite{Lejay:2007cs,Lejay:2008bt,ElMakrini:2007im,DelMoral:2003bn,Rousset:2006ih}.
However, these studies, some of which use branching particles
processes like Fleming-Viot (discussed below), only yield $\lambda_1$.

In a forthcoming work, 
we explore a mechanism for
computing $\lambda_2 - \lambda_1$ using observables.  The idea stems
from calculations in Theorem \ref{t:decorr_err_soft}, that, for an
observable $\mathcal{O}(x)$, as $t\to \infty$,
\begin{equation}
  \E^{\mu_0}\bracket{\mathcal{O}(X_t)\mid T> t} = \int_W \mathcal{O}(x)d\nu(x) +
  C(\mu_0, \mathcal{O})
  e^{-(\lambda_2-\lambda_1)t} + \ldots
\end{equation}
In principle, $\lambda_2-\lambda_1$ could be extracted from a time
series of $\E^{\mu_0}\bracket{\mathcal{O}(X_t)\mid T> t}$.  This
introduces a variety of questions, such as what observables to use and
how to perform such a fitting.  Thus, we will have a method for
dynamically estimating $\tcorr$ and $\tphase$.

\subsection{Dephasing Mechanism}

The efficiency of our dephasing algorithm can be improved by the
availability of multiple processors.  For instance, assume we have $N$
processors available for the replicas and that $k$ replicas have
successfully been run until $\tcorr$.  We are still waiting for $N-k$
replicas to successfully dephase.  Rather than let $k$ processors sit
idle, they could record the successful replicas, and run independent
realizations.  As more replicas finish dephasing, more processors can
be brought to bear on the outstanding replicas.

In practice, as replicas are deemed to have been successfully
dephased, they are promoted to the parallel step, \cite{Perez:fk}.
Thus, there is no bottleneck at the dephasing step 
from waiting to get $N$ realizations dephased.

There are other approaches to dephasing too, such as  Fleming-Viot or
Moran branching interacting particle processes,
\cite{Bieniek:2011jf,Bieniek:2012jg,Grigorescu:2004bs,DelMoral:2011wi}.
These merit consideration for ParRep.  These approaches, which
randomly split a surviving process
every time another process exits the well, can provide additional
information, such as an estimate of $\lambda_1$.  Moreover, no
processor sits idle at anytime.  However, two challenges are
introduced.  On a practical level, one needs to implement additional
communication routines and synchronization across the processors to
request and send configurations as trajectories are killed.  The
second challenge is analytical, as the dephased processes will now be
only approximately independent.  This complicates the analysis of the
how the error in the dephasing step cascades through the parallel
step.

\subsection{Other Challenges}
Another task is to assess the cumulative error over many ParRep
cycles.  The hitting point distribution will be perturbed by the
algorithm, meaning that the sequence in which the states are visited
would also be perturbed.  Quantifying the error across many steps, and
showing that it may be made small, would complete the justification of
ParRep over the lifetime of a simulation.  But to begin such a study,
one must decide how to measure
\[
\dist(\mathcal{S}_t, \mathcal{S}_t^{\rm ParRep}).
\]
The problem is $\mathcal{S}_t$ is not a Markovian process.  A particle
that sits near the edge of the well is likely to exit much sooner than
one which is near the minima of the well.  But that information is
lost in the coarse graining.  Knowing how long $X_t$ has been in the
well provides some amount of information; it tells us the proximity to
the QSD, from which we can get an exponential exit time.

Despite the challenge of studying the coarse grained flow, we can
report that ParRep appears to work as predicted over multiple wells.
Consider the flow
\begin{equation}
\label{e:periodic}
dX_t = -2 \pi \sin(\pi X_t) dt + \sqrt{2} dB_t.
\end{equation}
For this equation, with initial condition $X_0 = 0$, we examined the time it
would take to reach the wells centered at $x = \pm 10$.  In other
words, we sought to compute
\[
T_{\pm 10} = \inf\set{t\mid \abs{X_t} \geq 9}.
\]
For this problem, we ran the full ParRep algorithm (decorrelation,
dephasing and parallel steps) within {\it each} well.  During
dephasing, the replicas were initiated from the minima of the present
well, $0, \pm 2\pi, \pm 4\pi, \ldots$  We ran 10000 realizations of this experiment, varying
$k_\corr$ and $k_\phase$, where
\begin{equation}
\label{e:multiple_params}
\tcorr  = \frac{k_\corr}{\lambda_2 - \lambda_1}, \quad \tphase  =
\frac{k_\phase}{\lambda_2 - \lambda_1}.
\end{equation}
Since the wells are periodic, we can use spectral methods to compute $\lambda_1
\approx .202280$ and $\lambda_2 \approx 16.2588$ once, and
we then have these values for all the wells.  The results, with
$\Delta t = 10^{-4}$ and $N=100$ replicas, appear in Figure \ref{f:manywell}

\begin{figure}
\includegraphics[width=10cm]{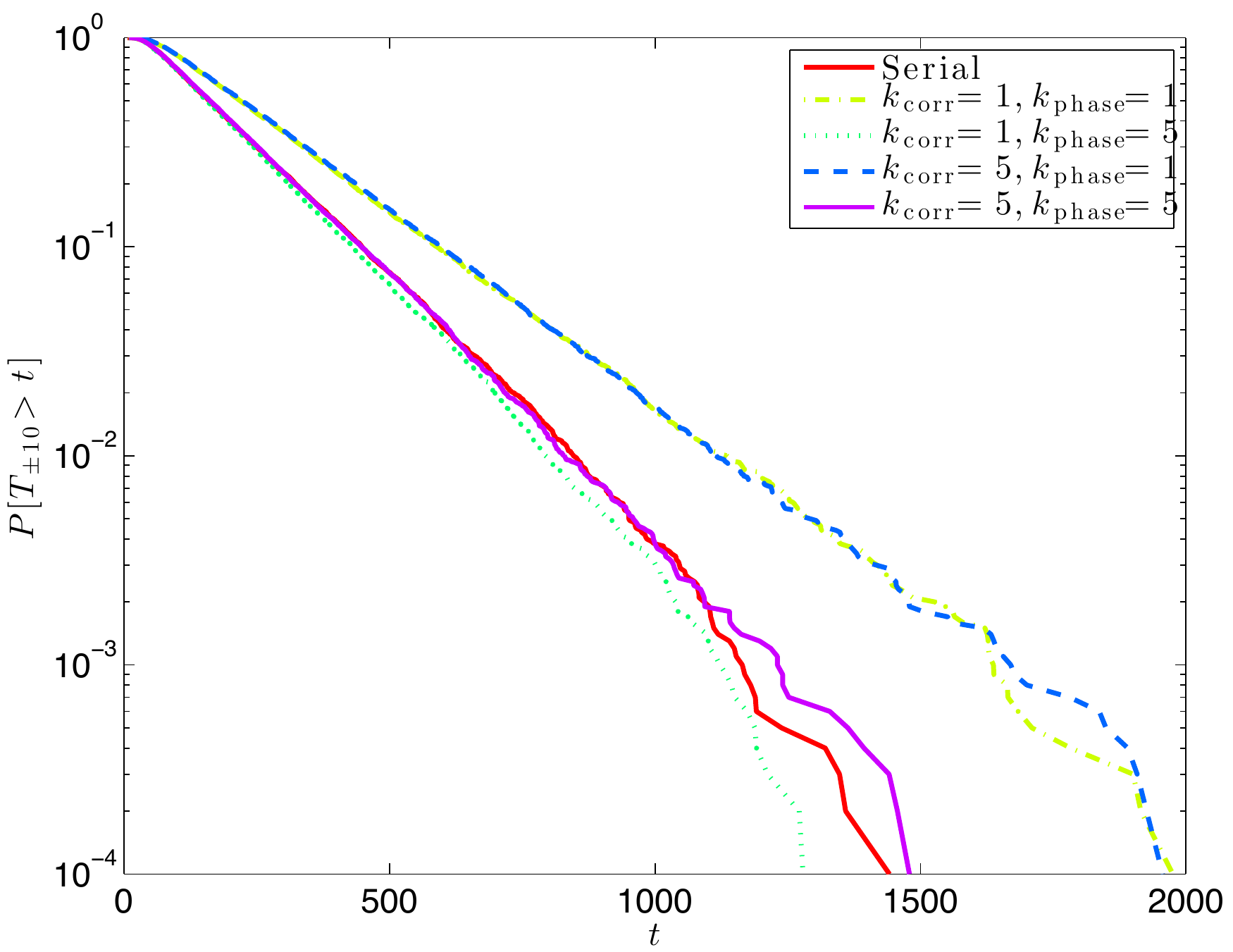}
\caption{The cumulative distribution for the time for it takes
  trajectory \eqref{e:periodic} to reach the wells centered at $\pm
  10$.  10000 realizations of each case were run with time step
  $\Delta t = 10^{-4}$.  $\tcorr$ and $\tphase$ relate to $k_\corr$
  and $k_\phase$ via \eqref{e:multiple_params}.  As expected, larger
  values of $\tcorr$ and $\tphase$ give better agreement with an
  unaccelerated process.}
\label{f:manywell}
\end{figure}

As we expect, for sufficiently large values of $\tcorr$ and $\tphase$,
the distributions agree with the serial process.  Indeed, in the cases
$k_\corr = k_\phase = 5$ and $k_\corr =1$, $k_\phase =5$, the exit
times agree with the serial realization at 5\% significance level
under a Kolmogorov-Smirnov test.  In addition, this experiment also
supports our calculations that, through the dephasing error, the
total error should be magnified by $N$ since increasing the
dephasing time improves the fit much more than increasing the
decorrelation time does.

Finally, we remark that we have only analyzed the continuous in time
problem, though we are ultimately interested in the associated
discrete in time algorithm.  Much of the analysis carries over to the
discrete in time case.  A discrete in time quasistationary
distribution exists, and there are extensive results on using
interacting particle algorithms for dephasing,
\cite{DelMoral:2004cv,DelMoral:2011wi}.  As in the continuous in time
case, there remains the subtlety of how to analyze the parallel step
when the dephased ensemble is only approximately independent.

However, the discrete time step introduces other subtleties.  Assume
one uses Euler-Maruyama time discretization with time step $\Delta
t$, and define the exit time as
\begin{equation}
  T^{\Delta t} =\inf\set{t_n\mid X_{t_n}\notin W}.
\end{equation}
For a uniform time step, we see that with no acceleration of the
dynamics, the exit times are integer multiples of $\Delta t$.  For ParRep,
this remains true for exits that take place during the decorrelation
step.  But for exit times taking place during the parallel step, the
exit times will be determined by multiples of $N \Delta t$.  With a
large number of processors, this effective time step could be quite
large. When comparing against the continuous in time problem, the
error of discretization could be magnified in ParRep.  In the
preceding experiment, $N\Delta t = .01$, which is small relative
to the exit time scale ($1/\lambda_1\approx 4.9$) and the
decorrelation time scale ($1/(\lambda_2 - \lambda_1) \approx .062$). Clearly, the
discrete in time case warrants a thorough investigation.

\appendix

\section{Summation Bounds}
\label{s:sum_bounds}

Much of our analysis relies on bounding series solutions,
\eqref{e:series_soln}, of \eqref{e:fk_parabolic}, to obtain
information about $X_t$ through the Feynman-Kac equation,
\eqref{e:fk}.  The key estimates needed in our work stem from Weyl's
Law for $L$:
\begin{prop}[Weyl's Law for $L$]
  \label{p:weylmu}
  There exist positive constants $c_1$ and $c_2$, independent of $k$,
  such that the eigenvalues of \eqref{e:linop} satisfy
  \begin{equation}
    \label{e:weyllaw}
    c_1 k^{{2}/{n}} \leq \lambda _k \leq c_2 k^{{2}/{n}}.
  \end{equation}
\end{prop}
Recall that $n$ denotes the dimension of the underlying problem;
$X_t\in \R^n$.

\begin{proof}
  We will not reproduce the proof here, which is accomplished by
  rewriting the eigenvalue problem as
  \begin{equation}
    \label{e:Ldivform}
    - \beta^{-1} \nabla\cdot  \paren{e^{-\beta V} \nabla u} = \lambda e^{-\beta V} u.
  \end{equation}
  This is justified because $V$ is smooth and $W$ is bounded; thus
  $e^{-\beta V}$ is smooth and nondegenerate.  This is now in the form
  of Theorem 6.3.1 of \cite{davies1996spectral} on Weyl's Law,
  yielding the result.
\end{proof}

Using Weyl's Law, we have our main summation result,
\begin{prop}
  \label{p:sum_bound2}
  Given $s\geq 0$, let $\mathbf{a} = (a_1, a_2, \ldots)$ satisfy
  \[
  \set{\sum_{k=1}^\infty \lambda_k^{-s}\abs{a_k}^2 }^{1/2} =
  \norm{\mathbf{a} }_{H^{-s}_\mu}<\infty.
  \]
  Let $f$ be defined as
  \begin{equation}
    f(\tau) \equiv \sum_{k=1}^\infty a_k \lambda_k^\alpha e^{-\tau \lambda_k}.
  \end{equation}
  For $ a > 0$, we have:
  \begin{enumerate}
  \item \begin{equation} \sup_{\tau\geq a} \abs{f(\tau)}\lesssim
      a^{-n/4 - \max\set{s/2 + \alpha, 0}} \norm{\mathbf{a}
      }_{H^{-s}_\mu} < \infty;
    \end{equation}
  \item The convergence of the series is uniform in $\tau \geq a$;
  \item $f$ is continuous.
  \end{enumerate}
\end{prop}
To prove Proposition \ref{p:sum_bound2}, we first have the following
lemma.
\begin{lem}
  \label{l:sum_bound}
  Let $\lambda_k$ be the eigenvalues and eigenfunctions of $L$,
  \eqref{e:linop}.  There exists a constant $C>0$, independent of
  $\tau$, such that for all $\tau >0$,
  \begin{equation}
    \label{e:summation_bound}
    \sum_{k=1}^\infty \lambda_k^{\alpha} e^{-\tau
      \lambda_k}\leq C \tau^{-n/2 - \max\set{\alpha, 0}}.
  \end{equation}
\end{lem}
The reader should rightfully expect the lefthand side of
\eqref{e:summation_bound} to grow as $\alpha \to \infty$.  Indeed, the
constant $C$ depends on $\alpha$ and will grow.  However, as $\alpha$
is fixed, and we are interested in an estimate in $\tau$, this is
suppressed.
\begin{proof}
  For $\alpha \leq 0$,
  \begin{equation*}
    \begin{split}
      \sum_{k=1}^{\infty} e^{-\tau \lambda_k} \lambda_k^\alpha &\leq
      \sum_{k=1}^{\infty} e^{-\tau \lambda_k} \lambda_1^\alpha
      \leq \sum_{k=1}^{\infty} e^{-c_1 \tau  k^{2/n}} \lambda_1^\alpha\\
      & \leq \lambda_1^\alpha\int_0^\infty e^{-c_1 \tau k^{2/n}}dk =
      \lambda_1^\alpha (c_1 \tau)^{-n/2} \Gamma \bracket{1 +
        \frac{n}{2}}.
    \end{split}
  \end{equation*}
  In the above computation, we approximated the sum as the lower
  Riemann sum of the integral.

  For $\alpha > 0$, we begin by estimating
  \begin{equation*}
    \sum_{k=1}^{\infty} e^{-\tau \lambda_k} \lambda_k^\alpha\leq\sum_{k = 1}^{\infty} e^{-c_1\tau k^{2/n}}
    c_2^{\alpha} k^{2\alpha / n }.
  \end{equation*}
  For sufficiently large $k$,
  \[
  k \geq k_1 \equiv \left\lceil{\paren{\frac{\alpha}{ c_1 \tau}}^{n/
        2}}\right\rceil,
  \]
  the summand is monotonically decreasing, while for $k < k_1$, it is
  monotonically increasing.  Splitting the sum up,
  \begin{equation*}
    \begin{split}
      \sum_{k = 1}^{\infty} e^{-c_1\tau k^{2/n}} k^{2\alpha/ n} &=
      \sum_{k=1}^{k_1} e^{-c_1\tau k^{2/n}} k^{2\alpha/ n + 1 } +
      \sum_{k=k_1+1}^{\infty} e^{-c_1\tau k^{2/n}}
      k^{2\alpha/ n } \\
      & \leq e^{-c_1 \tau} \sum_{k=1}^{k_1} k^{2\alpha/ n }+
      \sum_{k=k_1+1}^{\infty} e^{-c_1\tau k^{2/n}} k^{2\alpha/ n }.
    \end{split}
  \end{equation*}
  Crudely bounding the first sum in terms of a $\max$, and treating
  the latter sum as a lower Riemann approximations of an integral,
  \begin{equation*}
    \begin{split}
      &\sum_{k = 1}^{\infty} e^{-c_1\tau k^{2/n}} k^{2\alpha/ n } \leq
      e^{-c_1 \tau} k_1 \cdot k_1 ^{2\alpha/ n }+ \int_{k_1}^\infty e^{-c_1\tau k^{2/n}} k^{2\alpha/ n  }dk \\
      &\quad\leq e^{-c_1 \tau} \bracket{\paren{\frac{\alpha}{ c_1
            \tau}}^{n/2} +
        1}^{2\alpha/ n +1 }+ \int_{0}^\infty  e^{-c_1\tau k^{2/n}} k^{2\alpha/ n }dk \\
      & \quad\leq \paren{\frac{c_1 \tau}{ \alpha }}^{-n/2-\alpha }
      e^{-c_1 \tau}\bracket{1 +\paren{\frac{c_1 \tau}{ \alpha }}^{n/2}
      }^{2\alpha/n +1}+
      \frac{n}{2}(c_1\tau)^{-n/2 - \alpha} \Gamma\bracket{\frac{n}{2}+\alpha}\\
      &\quad \lesssim \tau^{-n/2 -\alpha}.
    \end{split}
  \end{equation*}
\end{proof}
The integrals were computed using Mathematica, with the commands
\begin{verbatim}
Integrate[Exp[-c*t*k^(2/n)],{k,0,Infinity}]

Integrate[Exp[-c*t*k^(2/n)]*k^(2*a/n+1),{k,0,Infinity}]
\end{verbatim}

Now we prove Proposition \ref{p:sum_bound2}.
\begin{proof}
  We first observe that $f$ is well defined and bounded:
  \begin{equation*}
    \abs{f(\tau)} \leq \sum_{k=1}^\infty \abs{a_k} \lambda_k^\alpha
    e^{-\tau \lambda_k}\leq \set{\sum_{k=1}^\infty\lambda^{s+ 2\alpha }
      e^{- 2\lambda_k \tau }  }^{1/2}\norm{\mathbf{a}}_{H^{-s}_\mu}.
  \end{equation*}
  Applying Lemma \ref{l:sum_bound} with $\alpha\mapsto s + 2\alpha $
  and $\tau \mapsto 2 a$,
  \[
  \sum_{k=1}^\infty\lambda^{s + 2\alpha } e^{- 2 \lambda_k a }
  \lesssim (2a)^{-n/2 - \max\set{s+ 2\alpha , 0} }.
  \]

  To prove uniform convergence, let
  \[
  f_m(\tau) \equiv \sum_{k=1}^{m} a_k \lambda_k^\alpha e^{- \tau
    \lambda_k}
  \]
  denote the partial sum.  Obviously, each partial sum is continuous
  in $\tau$.  Then
  \begin{equation*}
    \begin{split}
      \abs{f(\tau)-f_m(\tau)} &\leq \sum_{k={m+1}}^\infty \abs{a_k}
      \lambda_k^\alpha e^{-\tau \lambda_k}\\
      & \leq \set{\sum_{k=m+1}^\infty \lambda_k^{s + 2\alpha }
        e^{-2\lambda_k a} }^{1/2}
      \norm{P_{[m+1,\infty)}\mathbf{a}}_{H_\mu^{-s}}\\
      &\leq \norm{\mathbf{a}}_{H_\mu^{-s}}\set{\sum_{k=m+1}^\infty
        \lambda_k^{s+ 2\alpha} e^{-2\lambda_k a} }^{1/2}.
    \end{split}
  \end{equation*}
  Examining the sum,
  \begin{equation*}
    \begin{split}
      \sum_{k=m+1}^\infty \lambda_k^{s+ 2\alpha} e^{-2\lambda_k a} &
      \lesssim \sum_{k=m+1}^\infty k^{2s/n + 4 \alpha /n} e^{-2 c_1 a
        k^{2/n}}.
    \end{split}
  \end{equation*}
  Taking $m$ sufficiently large, the summand will be strictly
  decreasing in $k$, so we can treat it as a lower Riemann sum for the
  integral
  \[
  \int_{m}^\infty k^{2s/n + 4 \alpha /n} e^{-2c_1 a k^{2/n}} dk.
  \]
  Changing variables by letting $k^{2/n} = l$,
  \[
  \sum_{k=m+1}^\infty \lambda_k^{s+ 2\alpha} e^{-2\lambda_k a}\lesssim
  \int_{m^{2/n}}^\infty l^{s+ 2\alpha + n/2 -1} e^{- 2c_1 a l} dl.
  \]
  If ${s + 2\alpha + n/2 -1} \leq 0$, then
  \[
  \sum_{k=m+1}^\infty \lambda_k^{s + 2\alpha} e^{-2\lambda_k
    a}\lesssim \int_{m^{2/n}}^\infty e^{- 2c_1 a l} dl = \frac{1}{2c_1
    a}e^{-2 m^{2/n} c_1 a}.
  \]
  On the other hand, if ${s+ 2\alpha + n/2 -1}> 0$, we can trade some
  of the exponential decay to eliminate the algebraic term,
  \[
  \sum_{k=m+1}^\infty \lambda_k^{s + 2\alpha} e^{-2\lambda_k
    a}\lesssim \int_{m^{2/n}}^\infty e^{- c_1 a l} dl = \frac{1}{c_1
    a}e^{-m^{2/n} c_1 a}.
  \]
  In either case, we see that for any $a>0$,
  \[
  \lim_{m\to \infty}\sup_{\tau \geq a}\abs{f(\tau) - f_m(\tau)} = 0.
  \]
  Since the partial sums converge uniformly to $f$, it is now a
  classical result to conclude that $f$ is continuous for $\tau \geq a
  > 0$, \cite{rudin1976principles}.
\end{proof}

\bibliographystyle{plain}

\bibliography{acc_dyn}

\end{document}

%% file: gsmacros.tex
\newcommand{\abs}[1]{\left| #1\right|}
\newcommand{\norm}[1]{\left\|#1\right\|}
\newcommand{\paren}[1]{\left(#1\right)}
\newcommand{\bracket}[1]{\left[#1\right]}
\newcommand{\set}[1]{\left\{#1\right\}}

\newcommand{\bigo}{\mathrm{O}}

\newcommand{\eps}{\epsilon}

\newcommand{\grad}{\nabla}

\newcommand{\dist}{\mathrm{dist}}

\newcommand{\refe}{\mathrm{ref}}
\newcommand{\simu}{\mathrm{sim}}
\newcommand{\phase}{\mathrm{phase}}

\newcommand{\corr}{\mathrm{corr}}

\newcommand{\ser}{\mathrm{s}}
\newcommand{\para}{\mathrm{p}}

\newcommand{\Texit}{T}

\newcommand{\tphase}{t_{\mathrm{phase}}}
\newcommand{\tcorr}{t_{\mathrm{corr}}}
\newcommand{\tsimu}{t_{\mathrm{sim}}}

\newcommand{\inner}[2]{\left\langle #1,#2\right\rangle}

\newcommand{\E}{\mathbb{E}}
\newcommand{\prob}{\mathbb{P}}

\newcommand{\calP}{\mathcal{P}}

\newcommand{\R}{\mathbb{R}}

\newcommand{\supp}{\mathrm{supp}\:}